\documentclass[journal]{IEEEtran}

\usepackage{amsmath,amssymb,amsfonts}
\usepackage{algorithmic}
\usepackage{graphicx}
\usepackage{algorithm,algorithmic}
\usepackage{hyperref}
\usepackage{amsthm}
\usepackage{subcaption}
\usepackage{mathtools}
\usepackage{xcolor}

\usepackage{titlesec}

\newtheorem{theorem}{Theorem}%{Theorem}%[section]
\newtheorem{lemma}{Lemma}

\newtheorem{assumption}{Assumption}
\newtheorem{remark}{Remark}

\ifCLASSINFOpdf
  
\else
  
\fi

\begin{document}

\title{Bare Demo of IEEEtran.cls\\ for IEEE Journals}

\title{Provably Convergent Decentralized Optimization over Directed Graphs under Generalized Smoothness}
\author{Yanan Bo and Yongqiang Wang, ~\IEEEmembership{Senior Member,~IEEE} \thanks{The work was supported in part by the National Science Foundation under Grants CCF-2106293, CCF-2215088, CNS-2219487, CCF-2334449, and CNS-2422312 (Corresponding author: Yongqiang Wang).} \thanks{The authors are with the Department of Electrical and Computer Engineering, Clemson University, Clemson, SC 29634, USA (e-mail: ybo@clemson.edu; yongqiw@clemson.edu).} }

\markboth{ }%
{Shell \MakeLowercase{\textit{et al.}}: Bare Demo of IEEEtran.cls for IEEE Journals}

\maketitle

\begin{abstract}
Decentralized optimization has become a fundamental tool for large-scale learning systems; however, most existing methods rely on the classical Lipschitz smoothness assumption, which is often violated in problems with rapidly varying gradients. Motivated by this limitation, we study decentralized optimization under the generalized $(L_0, L_1)$-smoothness framework, in which the Hessian norm is allowed to grow linearly with the gradient norm, thereby accommodating rapidly varying gradients beyond classical Lipschitz smoothness. We integrate gradient-tracking techniques with gradient clipping and carefully design the clipping threshold to ensure accurate convergence over directed communication graphs under generalized smoothness. In contrast to existing distributed optimization results under generalized smoothness that require a bounded gradient dissimilarity assumption, our results remain valid even when the gradient dissimilarity is unbounded, making the proposed framework more applicable to realistic heterogeneous data environments. 
We validate our approach via numerical experiments on standard benchmark datasets, including LIBSVM and CIFAR-10, using regularized logistic regression and convolutional neural networks, demonstrating superior stability and faster convergence over existing methods.

% Decentralized optimization has become a fundamental tool for large-scale learning systems, yet most existing methods rely on the classical Lipschitz smoothness assumption, which is often violated in problems with rapidly varying gradients.
% Motivated by this limitation, we study decentralized optimization under the generalized $(L_0, L_1)$-smoothness framework, where the Hessian norm is allowed to grow linearly with the gradient norm, thereby accommodating rapidly varying gradients beyond classical Lipschitz smoothness.
% % We identify that existing decentralized methods impose restrictive uniform gradient dissimilarity bounds, which are incompatible with generalized smoothness and heterogeneous agents.
%  In this work, we proposed a decentralized optimization algorithm over directed communication graphs under generalized smoothness. Moreover, we relax the commonly used bounded gradient dissimilarity assumption that appears in most existing decentralized optimization studies, thereby making our framework more applicable to realistic heterogeneous data environments.  By carefully designing algorithmic
% parameters and developing refined inequalities tailored
% to our update structure, we established the convergence guarantees under this framework and showed that the proposed method achieves an 
% $\mathcal{O}(\epsilon^{-2})$ convergence rate, which matches the best-known complexity bound of centralized algorithms under the same smoothness condition. Extensive experiments on benchmark datasets validate the effectiveness and efficiency of the proposed algorithm.
\end{abstract}


\IEEEpeerreviewmaketitle

\section{Introduction}
Decentralized optimization has emerged as a fundamental paradigm for large-scale systems in which data and computation are distributed across multiple agents. 
Such settings naturally arise in a wide range of applications, including distributed machine learning~\cite{dean2012large}, multi-agent coordination and control~\cite{chen2019control}, sensor and communication networks~\cite{rabbat2004distributed}, smart grids and power systems~\cite{mohsenian2010optimal}, as well as modern cloud and content-delivery infrastructures~\cite{chen2015efficient, apostolopoulos2002multiple}. 
In these scenarios, agents collaboratively minimize a global objective while preserving data locality and operating without a central coordinator.

Despite significant progress, existing theoretical foundations of decentralized optimization are largely built upon the classical Lipschitz smoothness assumption, which requires the gradient of each local objective to be globally Lipschitz continuous.
This assumption, however, can be overly restrictive in modern large-scale and nonconvex learning problems, particularly those involving deep neural networks~\cite{nedic2009distributed, yuan2016convergence, chen2012diffusion, you2025stochastic, bo2024quantization, sayed2014adaptive}.
In fact, empirical evidence from neural network training indicates that the Hessian norm often scales with the gradient norm of the loss function, indicating that gradients may vary rapidly along the optimization trajectory and thereby violate standard smoothness conditions.

A more general and realistic smoothness framework, known as \emph{$(L_0, L_1)$-smoothness}, was introduced~\cite{zhang2019gradient}.
Under this condition, a differentiable function $g$ satisfies
\[
\|\nabla^2 g(\boldsymbol{\theta})\| \leqslant L_0 + L_1 \|\nabla g(\boldsymbol{\theta})\|,
\quad \text{for all } \boldsymbol{\theta} \in \mathbb{R}^{d},
\]
which allows the Hessian norm to grow with the gradient norm and strictly generalizes the classical notion of $L_0$-smoothness, recovered as the special case $L_1=0$.
This framework encompasses a broader class of functions, including polynomial and exponential functions~\cite{li2023convex}, and provides a more accurate characterization of the loss landscapes encountered in modern machine learning models such as LSTMs~\cite{zhang2019gradient} and Transformers~\cite{crawshaw2022robustness}.

In this work, we consider the following decentralized optimization problem under generalized smoothness
\begin{align}\label{Equ:P1}
\min_{\boldsymbol{\theta}\in \mathbb{R}^{d}} 
\; F(\boldsymbol{\theta})
= \frac{1}{N}\sum_{i=1}^{N} f_i(\boldsymbol{\theta}),
\end{align}
where each local objective $f_i$ is privately held by agent $i$ and is $(L_0, L_1)$-smoothness. Because agents only have access to local information, solving~\eqref{Equ:P1} requires coordination through local computation and information exchange over a communication network.

Since its introduction, the $(L_0, L_1)$-smoothness framework has attracted growing attention in optimization and learning research, with most existing results focused on centralized settings.
Zhang et al.~\cite{zhang2019gradient, zhang2020improved} employed the $(L_{0}, L_{1})$-smoothness condition to theoretically explain the acceleration effect of clipped SGD compared to standard SGD. Subsequent work has extended these results to different algorithmic variants and optimization contexts, including accelerated gradient methods~\cite{gorbunov2024methods, vankov2024optimizing}, clipped SGD with momentum~\cite{zhang2020improved}, normalized gradient descent with momentum~\cite{hubler2024parameter, chen2023generalized}, differentially private SGD~\cite{yang2022normalized}, generalized SignSGD~\cite{crawshaw2022robustness}, AdaGrad-Norm/AdaGrad~\cite{faw2023beyond, wang2023convergence}, Adam~\cite{li2023convergence}, $(L_{0}, L_{1})$-Spider~\cite{reisizadeh2025variance} and distributionally robust optimization~\cite{zhang2025revisiting}. 

In contrast, state-of-the-art results on decentralized optimization under the $(L_{0}, L_{1})$-smoothness condition remain very limited. Recently, Jiang et al.~\cite{jiang2025decentralized} proposed a first-order method for decentralized optimization under relaxed smoothness assumptions, but it requires global averaging at each iteration and is therefore not fully decentralized. Luo et al.~\cite{luo2025decentralizedstochasticnonconvexoptimization} and Sun~\cite{sunclipping} developed algorithms based on decentralized gradient descent combined with gradient normalization or clipping to address generalized smoothness. Nevertheless, these methods depend on the bounded gradient dissimilarity condition that substantially restricts their applicability in heterogeneous settings. Specifically, they assume that
\begin{equation}
    \|\nabla f_i(\boldsymbol{\theta}) - \nabla F(\boldsymbol{\theta})\| \leq \hat{b},
\end{equation}
holds for some constant $\hat{b}>0$~\cite{jiang2025decentralized, assran2019stochastic, 00c5a69517e340f28e941e91c58a6fb7, lian2017can, woodworth2020minibatch}. In fact, this condition is also adopted in other existing results on decentralized $(L_0, L_1)$-smooth optimization~\cite{jiang2025decentralized, luo2025decentralizedstochasticnonconvexoptimization, sunclipping}. While analytically convenient, this condition can be overly restrictive in heterogeneous or large-scale networks, where agents may possess substantially different data distributions or model architectures, causing gradient discrepancies to scale with the gradient norm and rendering the assumption ineffective.

In this paper, we relax this condition and allow gradient dissimilarity to be unbounded. Specifically, we assume
\begin{equation}
    \| \nabla f_i(\boldsymbol{\theta}) - \nabla F(\boldsymbol{\theta}) \|
    \leqslant (\ell-1) \| \nabla F(\boldsymbol{\theta}) \| + b,\footnote{The parameterization $(\ell-1)$ is adopted here for notational convenience in subsequent proofs.}
\end{equation}
where $\ell $ and $b$ are constants satisfying $\ell \geqslant 1$ and $b > 0$.
This relaxation naturally allows gradient discrepancies to scale with the gradient norm, thereby enabling broader applicability.

In addition, different from existing $(L_{0}, L_{1})$-smoothness results for decentralized optimization that rely on symmetric communication networks, we consider communication networks that can be asymmetric, which may be caused by asymmetric information flow, unidirectional links, or heterogeneous transmission capabilities~\cite{gharesifard2010does, sabattini2014decentralized}. 
It is worth noting that optimization over directed graphs is substantially more involved than over undirected graphs (see, e.g., \cite{nedic2009distributed, yuan2016convergence, chen2012diffusion, sundhar2010distributed, xin2020variance, pu2021distributed}), as asymmetric communication leads to biased information mixing, the absence of doubly stochastic weights, and nontrivial error propagation~\cite{Xi2015OnTD}.
In fact, all existing distributed optimization results for directed graphs, including SGP~\cite{nedic2014distributed}, SONATA~\cite{scutari2019distributed}, Push-DIGing/ADDOPT~\cite{nedic2017achieving}, and Push–Pull/AB~\cite{pu2020push}, rely on the standard Lipschitz smoothness assumption. Their extension to more general smoothness regimes that accommodate rapidly varying and heterogeneous gradients remains largely unexplored.

In this work, we address distributed optimization over directed graphs under generalized smoothness without requiring bounded gradient dissimilarity. Unlike~\cite{sunclipping}, which applies clipping to the raw local gradients and thus requires a bounded dissimilarity condition, our algorithm applies clipping to a local estimate of the global gradient. This approach allows us to effectively handle both the substantial discrepancies among local objectives and the additional imbalance caused by directed communication. In doing so, we bridge the aforementioned gaps and establish provable convergence guarantees under generalized and practically relevant conditions.

Our main contributions are summarized below:
\begin{itemize} 
% \item 
% In this paper, we study decentralized optimization under a \emph{generalized smoothness} framework, which naturally captures learning problems with rapidly varying gradient magnitudes.
% While generalized smoothness has been shown to be crucial, extending it to decentralized networks is particularly challenging.
% The difficulty arises from the fact that generalized smoothness allows gradient norms to vary significantly across agents, which fundamentally complicates the control of consensus errors and gradient heterogeneity in decentralized updates.
% As a result, decentralized optimization under generalized smoothness has remained largely unexplored in the existing literature.
% To cope with these challenges, the few existing decentralized methods under generalized smoothness typically rely on a \emph{uniformly bounded gradient dissimilarity} assumption to artificially control agent heterogeneity~\cite{jiang2025decentralized,luo2025decentralizedstochasticnonconvexoptimization,sunclipping}.
% However, such an assumption is often unrealistic and overly restrictive, especially under generalized smoothness, where gradient norms may vary dramatically throughout the optimization process.
% Motivated by this limitation, we introduce a more general \emph{relative gradient dissimilarity} condition, under which the discrepancy between local and global gradients scales with the magnitude of the global gradient.
% This relaxation broadens the applicability of decentralized optimization under generalized smoothness.
\item 
We propose a decentralized optimization algorithm with provable convergence under generalized $(L_0, L_1)$-smoothness without requiring bounded gradient dissimilarity---a property that, to our knowledge, has not been achieved previously. Decentralized optimization under generalized smoothness is fundamentally different from existing work under classical Lipschitz smoothness (e.g., \cite{nedic2009distributed, yuan2016convergence, nedic2017achieving, pu2020push, li2020revisiting, yuan2022revisiting, song2024optimal}), because generalized smoothness permits rapid, unbounded variation in gradient norms across agents. This variation substantially amplifies both consensus errors and gradient heterogeneity in decentralized optimization, posing significant challenges for convergence analysis. Unlike existing methods~\cite{jiang2025decentralized,luo2025decentralizedstochasticnonconvexoptimization,sunclipping} that rely on a bounded gradient dissimilarity assumption to handle generalized smoothness, our approach allows the gradient dissimilarity to be unbounded, reflecting more realistic heterogeneous scenarios.

\item
A key contribution of this work is a novel algorithmic design that enables accurate convergence under generalized smoothness conditions. Unlike~\cite{sunclipping}, which clips local gradients directly, our method applies clipping to local estimates of the global gradient, allowing us to remove the bounded gradient-dissimilarity assumption. This design is highly nontrivial, as naive clipping of local gradients can amplify discrepancies among agents and severely hinder convergence. Consequently, the proposed algorithm effectively manages both substantial heterogeneity among local objective functions and the additional imbalance induced by directed communication.

\item Another major contribution of this work lies in the development of new proof techniques. The combination of clipping and local gradient estimation introduces nonlinear, state-dependent perturbations, which prevent the use of conventional convergence analyses based on Lipschitz gradients. To address this, we establish a new theoretical framework by carefully designing algorithmic parameters and deriving refined inequalities tailored to our update structure. This approach enables us to provide the first convergence guarantee for decentralized optimization that simultaneously account for gradient clipping, directed communication networks, and generalized smoothness.
In fact, we prove that the algorithm converges to an $\epsilon$-stationary point
within $\mathcal{O}(1/\epsilon^2)$ iterations, matching the complexity bound of centralized algorithms under the same smoothness condition~\cite{zhang2019gradient}. 
% \item
% We establish the \emph{first convergence guarantees} for decentralized optimization algorithms that jointly incorporate gradient tracking and gradient clipping under generalized smoothness, by carefully designing algorithmic parameters and developing refined inequalities tailored to our update structure. 
% Despite the nonlinear and state-dependent perturbations introduced by clipping, our analysis accommodates directed communication graphs and asymmetric information flow. 
% Specifically, we show that the proposed algorithm converges to an $\epsilon$-stationary point within $\mathcal{O}(1/\epsilon^2)$ iterations, matching the complexity of centralized methods under the same smoothness conditions~\cite{zhang2019gradient}.

\item We validated our approach through numerical experiments on benchmark datasets, including LIBSVM and CIFAR-10, using regularized logistic regression and convolutional neural networks. The results show that our algorithm achieves significantly improved stability and faster convergence compared to existing methods.
\end{itemize}

The paper is organized as follows.
Section~\ref{Sec:Problem Formulation and Preliminaries} introduces the problem formulation and assumptions. Section~\ref{Sec:Proposed Algorithm} presents the proposed algorithm. Section~\ref{Sec:main_results} establishes the main convergence results, with detailed proofs deferred to the Appendix. Section~\ref{Sec:experiments} provides numerical experiments, and Section~\ref{Sec:conclusion} concludes the paper.

Notations: Let $\boldsymbol{x}_i^k \in \mathbb{R}^d$ denote the local optimization variable of agent~$i$ at iteration~$k$, and define the collection of all local variables as $\boldsymbol{x}^k = [(\boldsymbol{x}_1^k)^\top; \cdots; (\boldsymbol{x}_N^k)^\top] \in \mathbb{R}^{N \times d}$. Similarly, let $\boldsymbol{y}_i^k \in \mathbb{R}^d$ be the local estimate of the global gradient, and $\boldsymbol{y}^k = [(\boldsymbol{y}_1^k)^\top; \cdots; (\boldsymbol{y}_N^k)^\top] \in \mathbb{R}^{N \times d}$ denote the stacked estimates of all agents. The collection of local gradients evaluated at the local variables is denoted as $\nabla f(\boldsymbol{x}^k) = [\nabla f_1^\top(\boldsymbol{x}_1^k); \cdots; \nabla f_N^\top(\boldsymbol{x}_N^k)] \in \mathbb{R}^{N \times d}$.
For ease of analysis, we define the agent $i$'s effective stepsize at iteration $k$ after clipping as ${\alpha}_i^k = \alpha \min \{ 1, {c_0}/{\|\boldsymbol{y}_i^k\|} \}$, 
where $\alpha$ and $c_0$ are some constants. Note that $\alpha_i^k$ varies across agents and iterations. The scaled gradient is denoted as $\boldsymbol{\alpha}^k \boldsymbol{y}^k = [({\alpha}_1^k \boldsymbol{y}_1^k)^\top; \cdots; ({\alpha}_N^k \boldsymbol{y}_N^k)^\top] \in \mathbb{R}^{N \times d}$. The global gradient evaluated at the averaged variable $\bar{\boldsymbol{x}}^k = \tfrac{1}{N}\sum_{i=1}^{N}\boldsymbol{x}_i^k$ is denoted by $\nabla F(\bar{\boldsymbol{x}}^k) \in \mathbb{R}^{d}$.

\section{Problem Formulation and Preliminaries}\label{Sec:Problem Formulation and Preliminaries}

We consider a decentralized network consisting of $N$ agents communicating over a directed graph. Each agent $i \in [N] \coloneqq \{1,2,\cdots, N\}$ maintains a local objective function $f_i:\mathbb{R}^d \rightarrow \mathbb{R}$, and the global objective is to minimize the average of all local objective functions, i.e.,
\begin{equation}\label{Equ:P}
\begin{aligned}
\min_{\boldsymbol{x}\in \mathbb{R}^{N \times d}} &f(\boldsymbol{x})=\textstyle \frac{1}{N}\sum_{i=1}^{N}f_i(\boldsymbol{x}_i), \,\, \\
\:{\rm s.t.}\: & \,\,\boldsymbol{x}_1=\boldsymbol{x}_2=\cdots=\boldsymbol{x}_{N},
\end{aligned}
\end{equation}
where \(\textstyle \boldsymbol{x}=[ \boldsymbol{x}_1^\top;  \boldsymbol{x}_2^\top; \ldots; \boldsymbol{x}_N^\top] \in \mathbb{R}^{N \times d}\). In this paper, the local objective functions and the global objective function can be nonconvex.

We begin by introducing the assumptions and properties required for the objective functions.

\begin{assumption}[Lower bounded objective]\label{Assum:lower_bound}
The global function $f$ is lower bounded, i.e.,
\[
    \underline{f} := \inf_{\boldsymbol{x}\in\mathbb{R}^{ N \times d}} f(\boldsymbol{x}) > -\infty.
\]
\end{assumption}

\begin{assumption}[$(L_0, L_1)$-smoothness]\label{Assum:Lipschitz}
Each local function $f_i(\cdot)$ is twice continuously differentiable and $(L_0^i, L_1^i)$-smooth, i.e., there exist constants $L_0^i, L_1^i > 0$ such that
\begin{equation}
    \| \nabla^2 f_i(\boldsymbol{\theta}) \| \leqslant  L_0^i + L_1^i \| \nabla f_i(\boldsymbol{\theta}) \|, 
    \quad \forall \boldsymbol{\theta} \in \mathbb{R}^d.
\end{equation}
\end{assumption}

Under Assumption~\ref{Assum:Lipschitz}, the Hessian norm of the local objective functions can grow linearly with the gradient norm. 
This generalizes the standard Lipschitz gradient assumption which corresponds to the case where $L_1^i=0$.

Assumption~\ref{Assum:Lipschitz} is satisfied by a wide range of practical objective functions. 
Empirical evidence from logistic regression, deep neural networks for image classification, and language modeling demonstrates that local smoothness grows approximately linearly with gradient norm during training~\cite{zhang2019gradient,jiang2025decentralized}.
This behavior fundamentally violates the conventional uniform Lipschitz smoothness assumption, yet is naturally captured by the $(L_0,L_1)$-smoothness condition.

% \begin{remark}\label{Remark:polynomial}
%     Let $f$ be the univariate polynomial function $f(x)=\sum_{t=1}^{T}a_tx^t$. When $t\geqslant3$, then $f$ is $(L_0, L_1)$-smooth for some $L_0$ and $L_1$ but not $L$-smooth.
% \end{remark}   
% \begin{remark}
%     Empirical results show that in language modeling tasks (e.g., LSTM training on PTB dataset)~\cite{zhang2019gradient}, deep image classification models, including CNNs and ResNets trained on MNIST~\cite{jiang2025decentralized} and CIFAR 10~\cite{zhang2019gradient}, and logistic regression~\cite{jiang2025decentralized}, the local gradient Lipschitz constant exhibits an approximately linear dependence on the gradient norm along the training trajectory,  satisfying the $(L_0,L_1)$-smoothness condition.
% \end{remark}
 The following lemmas 
summarize two key properties of $(L_0,L_1)$-smoothness that will play an important 
role in our convergence analysis.
\begin{lemma}[\cite{zhang2020improved}, Lemma A.3]\label{Lemma:pre3}
Let $g$ be $(L_0,L_1)$-smooth, and let $c>0$ be a constant. For any 
$\boldsymbol{\theta},\boldsymbol{\vartheta}\in\mathbb{R}^d$ such that
$\|\boldsymbol{\theta}-\boldsymbol{\vartheta}\| \leqslant c/L_1$, we have
\begin{equation}
\label{Eq:pre3}
\begin{aligned}
g (\boldsymbol{\theta})
&\leqslant 
g (\boldsymbol{\vartheta})
 + \big\langle \nabla g (\boldsymbol{\vartheta}),\, \boldsymbol{\theta}-\boldsymbol{\vartheta} \big\rangle 
\\
&\quad{}\; 
+ \frac{A L_0 + B L_1 \|\nabla g (\boldsymbol{\vartheta})\|}{2}\,
  \|\boldsymbol{\theta}-\boldsymbol{\vartheta}\|^2,
\end{aligned}
\end{equation}
where 
\[
A = 1 + e^{c} - \frac{e^{c}-1}{c},
\qquad 
B = \frac{e^{c}-1}{c}.
\]
\end{lemma}

\begin{lemma}[\cite{zhang2020improved}, Corollary A.4]\label{Lemma:pre4}
Let $g$ be $(L_0,L_1)$-smooth, and let $c>0$ be a constant. For any 
$\boldsymbol{\theta},\boldsymbol{\vartheta}\in\mathbb{R}^d$ such that
$\|\boldsymbol{\theta}-\boldsymbol{\vartheta}\|\leqslant  c/L_1$, it holds that
\begin{equation}\label{Eq:pre4}
\begin{aligned}
&\|\nabla g (\boldsymbol{\theta}) - \nabla g (\boldsymbol{\vartheta})\|\\
\leqslant & 
\Bigl( A L_0 + B L_1 \|\nabla g (\boldsymbol{\vartheta})\| \Bigr)\,
      \|\boldsymbol{\theta}-\boldsymbol{\vartheta}\|,
\end{aligned}
\end{equation}
where
\[
A = 1 + e^{c} - \frac{e^{c}-1}{c},
\qquad
B = \frac{e^{c}-1}{c}.
\]
\end{lemma}
Next, we discuss how to quantify the heterogeneity among local objectives.

In the convergence analysis of decentralized optimization, a common approach is to assume that the dissimilarity  among local gradients is uniformly bounded, i.e., for all $\boldsymbol{\theta}\in\mathbb{R}^d$,
\begin{equation}\label{Eq:uniform_grad_dis}
    \frac{1}{N}\sum_{i=1}^N \bigl\| \nabla f_i(\boldsymbol{\theta}) - \nabla F(\boldsymbol{\theta}) \bigr\|^2 \leqslant \hat{b}^2
\end{equation}
holds for some constant $\hat{b}>0$.
Such a condition is widely used in  distributed optimization~\cite{assran2019stochastic,00c5a69517e340f28e941e91c58a6fb7,lian2017can,woodworth2020minibatch}.

While the uniform bound in~\eqref{Eq:uniform_grad_dis} may hold when the heterogeneity among local objectives is mild, it becomes overly restrictive in more general settings, even under the conventional smoothness condition.  In fact, this assumption can be violated even for simple quadratic functions when local objectives have different curvatures~\cite{wang2022unreasonable}. The situation becomes even more problematic under generalized smoothness conditions, which commonly arise in heterogeneous or large-scale networks~\cite{li2020federated}. For example, bridge regression may employ the $L_q$-norm regularizer $r (\boldsymbol{\theta}) = \sum_{j=1}^d |\theta_j|^q$ with $q > 2$~\cite{fu1998penalized}. Such regularizers make the objective functions $(L_0, L_1)$-smooth but not $L$-smooth, since $\|\nabla^2 r(\boldsymbol{\theta})\| = \mathcal{O} (\|\boldsymbol{\theta}\|^{q-2})$ grows unboundedly as $\|\boldsymbol{\theta}\| \to \infty$.
When different agents adopt regularizers with different weights for the purpose of coping with non‑IID data~\cite{li2020federated}, learning personalized models~\cite{yang2021regularized}, or conducting  multi‑task learning~\cite{smith2017federated}, the difference in regularizers $\|\nabla r_i(\boldsymbol{\theta}) - \frac{1}{N}\sum_{j=1}^{N}\nabla r_j(\boldsymbol{\theta}
)\|$ 
grows polynomially in $\|\boldsymbol{\theta}\|$. Consequently, objective 
functions inherently violate~\eqref{Eq:uniform_grad_dis}.

Motivated by this limitation, we relax this bounded gradient dissimilarity condition and allow the difference between local and global gradients to scale with the magnitude of the global gradient, which better captures heterogeneous optimization landscapes.

% A commonly used condition in decentralized optimization is the bounded gradient dissimilarity assumption, as follows:

% \begin{assumption}[Uniformly bounded gradient dissimilarity]\label{Assum:grad_sim}
% There exists a constant $\hat{b}>0$ such that the gradient diversity is uniformly bounded, i.e., for any $\boldsymbol{x}\in\mathbb{R}^d$,
% \begin{equation}
%     \frac{1}{N}\sum_{i=1}^N{\|\nabla f_i(\boldsymbol{x}) - \nabla F(\boldsymbol{x})\| ^2}\leqslant \hat{b}^2,
% \end{equation}
% where $F(\boldsymbol{x}) = \frac{1}{N}\sum_{i=1}^N f_i(\boldsymbol{x})$.
% \end{assumption}

% Assumption~\ref{Assum:grad_sim} enforces a uniform bound on the variance of local gradients and 
% is widely adopted in the literature for analyzing distributed optimization under homogeneous settings \cite{assran2019stochastic,00c5a69517e340f28e941e91c58a6fb7,lian2017can,woodworth2020minibatch}.
% However, for generalized smoothness functions—especially nonlinear or high-order polynomials—
% this assumption can be overly restrictive. Even for the simple cubic case in Remark~\ref{Remark:polynomial}, 
% it fails to hold.   
% To accommodate heterogeneity, we introduce a more general assumption that allows the difference between local and global gradients to scale with the magnitude of the global gradient.

\begin{assumption}\label{Assum:gradient_dissimilarity}
There exist constants $\ell \geqslant 1$ and $b > 0$ such that the following inequality holds for any $\boldsymbol{\theta} \in \mathbb{R}^d$:
\begin{equation}\label{Eq:unbound_grad_dis}
    \| \nabla f_i(\boldsymbol{\theta}) - \nabla F(\boldsymbol{\theta}) \|
    \leqslant (\ell-1) \| \nabla F(\boldsymbol{\theta}) \| + b.
\end{equation}
\end{assumption}

Assumption~\ref{Assum:gradient_dissimilarity} generalizes the bounded gradient dissimilarity condition by 
allowing the deviation between local and global gradients to depend linearly on $\|\nabla F(\boldsymbol{\theta})\|$.  
It is easy to verify that (\ref{Eq:uniform_grad_dis}) is a special case of 
Assumption~\ref{Assum:gradient_dissimilarity} with $\ell=1$. In addition, one can verify that the bridge regression problem discussed above satisfy Assumption~\ref{Assum:gradient_dissimilarity} even when different agents using different $q$ in their regularizers. 

Moreover, under Assumption~\ref{Assum:Lipschitz}, we can prove that the global function 
$F(\boldsymbol{\theta}) = \tfrac{1}{N}\sum_{i=1}^{N} f_i(\boldsymbol{\theta})$ is also $(L_0, L_1)$-smooth, as detailed in the Lemma~\ref{Lemma:global_generalized} below, with its proof given in Appendix A.
% ~\ref{appendix:global_smoothness}. 

\begin{lemma}\label{Lemma:global_generalized}
    When every $f_i$ is $(L_0, L_1)$-smooth according to Assumption~\ref{Assum:Lipschitz}, we have that the global objective $F(\boldsymbol{\theta}) = \tfrac{1}{N}\sum_{i=1}^{N} f_i(\boldsymbol{\theta})$ is also $(L_0, L_1)$-smooth, with $L_0=\frac{1}{N}\sum_{i=1}^N{\left( L_{0}^{i}+L_{1}^{i}b \right)}$ and $L_1=\frac{\ell}{N}\sum_{i=1}^N{L_{1}^{i}}$. 
\end{lemma}

% \begin{remark}
% Consider the bridge regression problem with $L_{q}$-norm regularizers ($q > 2$) where 
% agents adopt varying coefficients $\lambda_i > 0$. Since the gradient dissimilarity 
% grows at most as fast as the average gradient norm (both scale as $O(\|x\|^{q-1})$), 
% the relaxed condition~\eqref{Eq:unbound_grad_dis} holds with appropriate constants $(\ell, b)$.
% \end{remark}

% \begin{remark}
% Consider $i=1,2,3$ with $f_i(x)=\tfrac{1}{i+2}x^{i+2}$. Then, 
% $\nabla f_i(x)=x^{i+1}$ and $F(x)=\tfrac{1}{3}\big(\tfrac{1}{3}x^3+\tfrac{1}{4}x^4+\tfrac{1}{5}x^5\big)$, 
% so that $\nabla F(x)=\tfrac{1}{3}(x^2+x^3+x^4)$.  
% For large $x$, the local gradients differ significantly, e.g.,
% \[
% \nabla f_3(x) - \nabla F(x) = x^4 - \tfrac{1}{3}(x^2 + x^3 + x^4)
% = \tfrac{2}{3}x^4 - \tfrac{1}{3}(x^2 + x^3),
% \]
% which grows unboundedly with $x$. Thus, there exists no constant $\hat{b}$ that can uniformly bound 
% $\|\nabla f_i(x)-\nabla F(x)\|$ for all $x$, and consequently, Assumption~\ref{Assum:grad_sim} fails.
% Nevertheless, it satisfies Assumption~\ref{Assum:gradient_dissimilarity} with 
% appropriate constants $\ell,b>0$ (e.g., $\ell=2$, $b=1$), since both sides grow polynomially in $x$.
% \end{remark}

Finally, we describe the assumptions on the underlying directed communication graph, which are described by two mixing matrices $\boldsymbol{R}$ and $\boldsymbol{C}$.

\begin{assumption}[Mixing matrices]\label{Assum:matrices_RC}
The matrix $\boldsymbol{R}\in\mathbb{R}^{N\times N}$ is nonnegative and row-stochastic ($\boldsymbol{R}\mathbf{1}=\mathbf{1}$), 
and the matrix $\boldsymbol{C}\in\mathbb{R}^{N\times N}$ is nonnegative and column-stochastic 
($\mathbf{1}^\top\boldsymbol{C}=\mathbf{1}^\top$).  
Both have positive diagonal entries.  
The $\boldsymbol{R}$-induced directed graph $\mathcal{G}_{\boldsymbol{R}}$ contains at least one spanning tree, and the $\boldsymbol{C}$-induced directed graph
$\mathcal{G}_{\boldsymbol{C}^\top}$ is strongly connected.
\end{assumption}

Under Assumption~\ref{Assum:matrices_RC}, we recall several results from~\cite{pu2020push} concerning the spectral properties of the mixing matrices.

\begin{lemma}[\cite{pu2020push}, Lemma 1]\label{Lemma:uv}
Under Assumption~\ref{Assum:matrices_RC}, the matrix $\boldsymbol{R}$ has a nonnegative left eigenvector $\boldsymbol{u}^\top$ 
(associated with eigenvalue $1$) satisfying $\boldsymbol{u}^\top \mathbf{1} = N$. Similarly, the matrix $\boldsymbol{C}$ has a 
strictly positive right eigenvector $\boldsymbol{v}$ (associated with eigenvalue $1$) satisfying $\mathbf{1}^\top \boldsymbol{v} = N$. Moreover, we have $\boldsymbol{u}^\top \boldsymbol{v} >0$.
\end{lemma}

\begin{lemma}[\cite{pu2020push}, Lemma 3]
Suppose that Assumption~\ref{Assum:matrices_RC} holds. Let $\rho_{R}$ and $\rho_{C}$ be the spectral radius of 
$\left(\boldsymbol{R} - \frac{1}{N} \boldsymbol{1}\boldsymbol{u} ^{\top}\right) 
\quad \text{and} \quad 
\left(\boldsymbol{C} - \frac{1}{N} \boldsymbol{v} \boldsymbol{1}^{\top}\right)$,
respectively. Then, we have 
$\rho_{R} < 1 
\quad \text{and} \quad 
\rho_{C} < 1$.
\end{lemma}
\begin{lemma}[\cite{pu2020push}, Lemma 4]
\label{Lemma:sigma_R,C}
There exist matrix norms $\|\cdot\|_{R}$ and $\|\cdot\|_{C}$ such that 
\[
\sigma_{R} := \left\| \boldsymbol{R} - \frac{1}{N} \boldsymbol{1}\boldsymbol{u} ^{\top} \right\|_{R} < 1, 
\quad 
\sigma_{C} := \left\| \boldsymbol{C} - \frac{1}{N} \boldsymbol{v} \boldsymbol{1}^{\top} \right\|_{C} < 1,
    \]
and $\sigma_{R}$ and $\sigma_{C}$ are arbitrarily close to $\rho_{R}$ and $\rho_{C}$, respectively.  

In addition, given any diagonal matrix $\boldsymbol{W} \in \mathbb{R}^{N \times N}$, we have 
\[
\|\boldsymbol{W}\|_{R} = \|\boldsymbol{W}\|_{C} = \|\boldsymbol{W}\|_{2}.
\]
\end{lemma}

We also recall the following norm-equivalence result:

\begin{lemma}[\cite{pu2020push}, Lemma 6]
\label{Lemma:norm-equivalence}
There exist constants $\delta_{C,R}, \delta_{C,2}, \delta_{R,C}, \delta_{R,2} > 0$ such that for all 
$\boldsymbol{\theta} \in \mathbb{R}^{d}$, we have
\[
\begin{aligned}
\|\boldsymbol{\theta}\|_{C} &\leqslant\delta_{C,R} \|\boldsymbol{\theta}\|_{R}, \quad
\|\boldsymbol{\theta}\|_{C} \leqslant\delta_{C,2} \|\boldsymbol{\theta}\|_{2}, \\
\|\boldsymbol{\theta}\|_{R} &\leqslant\delta_{R,C} \|\boldsymbol{\theta}\|_{C}, \quad
\|\boldsymbol{\theta}\|_{R} \leqslant\delta_{R,2} \|\boldsymbol{\theta}\|_{2}.
\end{aligned}
\]

In addition, with a proper rescaling of the norms $\|\cdot\|_{R}$ and $\|\cdot\|_{C}$, for all 
$\boldsymbol{\theta} \in \mathbb{R}^{d}$, we have
$
\|\boldsymbol{\theta}\|_{2} \leqslant \|\boldsymbol{\theta}\|_{R} $
and
$
\|\boldsymbol{\theta}\|_{2} \leqslant \|\boldsymbol{\theta}\|_{C}$.  
\end{lemma}

The assumptions and lemmas in this section are necessary to establish convergence of the proposed decentralized optimization algorithm  under directed communication graphs. 

\section{The Proposed Algorithm}\label{Sec:Proposed Algorithm}

In this section, we propose a new decentralized optimization algorithm that ensures accurate convergence under generalized smoothness conditions over directed graphs, even when the dissimilarity between agents’ gradients is unbounded. The basic idea is to apply gradient clipping to a local estimate of the global gradient by leveraging the gradient-tracking framework. To the best of our knowledge, this is the first work that integrates gradient clipping into gradient tracking to counteract the rapid growth of discrepancies between individual agents’ optimization variables induced by generalized smoothness.

It is worth noting that this integration introduces significant challenges in the convergence analysis. In particular, the introduction of clipping results in nonlinear, state-dependent perturbations, which preclude the direct application of conventional convergence analyses for gradient tracking that rely on Lipschitz gradient assumptions. To overcome these difficulties, we develop a new theoretical framework by carefully designing the algorithmic parameters and deriving refined inequalities tailored to our update structure. The proposed algorithm is summarized in Algorithm \ref{alg:1}, and the convergence analysis is presented in the next section.

% \begin{algorithm}[htpb]
%    \caption{Clipped-Gradient Tracking (CGT) Algorithm}
%    \label{alg:1}
%    \begin{algorithmic}[1]
%       \STATE \textbf{Initialization:} $\boldsymbol{x}_{i}^{0}\in\mathbb{R}^d$, $\boldsymbol{y}_{i}^{0} = \nabla f_i(\boldsymbol{x}_{i}^{0})$ for each agent $i$.
%       \STATE \textbf{Parameters:} stepsize $\alpha$, clipping threshold $c_0$, mixing matrices $\boldsymbol{R}$, $\boldsymbol{C}$.
%       \FOR{$k = 0,1,2,\dots$}
%          \STATE Each agent $i$ sends $\boldsymbol{x}_i^k$ to neighbors in $\mathcal{N}_{\mathbf{R},i}^{\mathrm{out}}$ 
%          and $\boldsymbol{y}_i^k$ to neighbors in $\mathcal{N}_{\mathbf{C},i}^{\mathrm{out}}$.
%          \STATE Upon receiving information from $\mathcal{N}_{\mathbf{R},i}^{\mathrm{in}}$, update:
%          \begin{equation}  \label{Eq:3}  
%            \boldsymbol{x}_i^{k+1} = \sum_{j\in \mathcal{N}_{\mathbf{R},i}^{\mathrm{in}}}\mathbf{R}_{ij}\boldsymbol{x}_j^k
%            - \alpha \min\!\left\{1, \frac{c_0}{\|\boldsymbol{y}_i^k\|}\right\}\boldsymbol{y}_i^k.
%          \end{equation}
%          \STATE Upon receiving $\boldsymbol{y}_j^k$ from $\mathcal{N}_{\mathbf{C},i}^{\mathrm{in}}$, update:
%          \begin{equation} \label{Eq:4}
%            \boldsymbol{y}_i^{k+1} = 
%            \sum_{j\in \mathcal{N}_{\mathbf{C},i}^{\mathrm{in}}}\mathbf{C}_{ij}\boldsymbol{y}_j^k 
%            + \big(\nabla f_i(\boldsymbol{x}_i^{k+1}) - \nabla f_i(\boldsymbol{x}_i^k)\big).
%          \end{equation}
%       \ENDFOR
%    \end{algorithmic}
% \end{algorithm}

\begin{algorithm}[htpb]
   \caption{Clipped-Gradient Tracking (CGT)}
   \label{alg:1}
   Choose stepsize $\alpha > 0$, clipping threshold $c_0 > 0$,
   
    \quad in-bound mixing weights $R_{ij} \geqslant 0$ for all $j \in \mathcal{N}_{\boldsymbol{R},i}^{\mathrm{in}}$,
    
   \quad and out-bound weights $C_{li} \geqslant 0$ for all $l \in \mathcal{N}_{\boldsymbol{C},i}^{\mathrm{out}}$;
   
   Each agent $i$ initializes with any arbitrary $\boldsymbol{x}_{i}^0 \in \mathbb{R}^d$ and 
   $\boldsymbol{y}_{i}^0 = \nabla f_i(\boldsymbol{x}_{i}^0)$;
   
   \textbf{for} $k = 0, 1, \ldots,$ \textbf{do}
   
   \quad for each $i \in [N]$,
   
   \quad\quad agent $i$ receives $\boldsymbol{x}_{j}^{k}$ from each $j \in \mathcal{N}_{\boldsymbol{R},i}^{\mathrm{in}}$;
   
   \quad\quad agent $i$ sends $C_{li}\boldsymbol{y}_{i}^{k}$ to each $l \in \mathcal{N}_{\boldsymbol{C},i}^{\mathrm{out}}$;
   
   \quad for each $i \in [N]$,
   \begin{align}
      \boldsymbol{x}_{i}^{k+1} &= \sum_{j=1}^{N} R_{ij}\boldsymbol{x}_{j}^{k} 
      - \alpha \min\!\left\{1, \frac{c_0}{\|\boldsymbol{y}_{i}^{k}\|}\right\}\boldsymbol{y}_{i}^{k} \label{Eq:3}\\
      \boldsymbol{y}_{i}^{k+1} &= \sum_{j=1}^{N} C_{ij}\boldsymbol{y}_{j}^{k} 
      + \nabla f_i(\boldsymbol{x}_{i}^{k+1}) - \nabla f_i(\boldsymbol{x}_{i}^{k}) \label{Eq:4}
   \end{align}
   
   \textbf{end for}
\end{algorithm}

The algorithm follows the standard gradient-tracking framework: in addition 
to maintaining a local optimization variable $\boldsymbol{x}_{i}^k$, each agent $i$ also maintains an 
auxiliary variable $\boldsymbol{y}_{i}^k$ that tracks the evolution of the global gradient. As 
discussed in~\cite{karimireddy2020scaffold}, the inclusion of this additional variable is crucial for 
ensuring accurate descent directions in decentralized optimization, particularly 
when the data across agents are heterogeneous.

At every iteration, each agent mixes its current optimization variable $\boldsymbol{x}_{i} ^k$ with those received from 
its in-neighbors through the row-stochastic matrix $\boldsymbol{R}$, while its gradient tracking
variable $\boldsymbol{y}_{i}^k$ is mixed using the column-stochastic matrix $\boldsymbol{C}$.

A fundamental difference from the conventional gradient-tracking framework is that we apply a clipping operation on the local tracking variable $\boldsymbol{y}_{i}^k$, which is necessary to suppress the rapid growth of agent discrepancies caused by the fast gradient variations permitted under $(L_0, L_1)$-smoothness.  Specifically, $\boldsymbol{y}_{i}^k$ is capped at $c_0$ when its norm exceeds $c_0$ and remains unchanged when its norm is below $c_0$.  
It is worth mentioning that in our algorithm, such clipping is applied to $\boldsymbol{y}_{i}^k$ rather than directly to the local gradient $\nabla f_i(\boldsymbol{x}_{i}^k)$ like~\cite{sunclipping}. We argue that this is important for us to obtain stronger results than~\cite{sunclipping} because local gradients may vary dramatically across agents in the heterogeneous setting, and clipping them directly would exacerbate the state discrepancies among the agents. In contrast, tracking variables serve as estimators of the global gradient, making them more stable quantities on which clipping can be performed without compromising convergence.

\section{Convergence Analysis}\label{Sec:main_results}

In this section, we rigorously establish that Algorithm~\ref{alg:1} ensures accurate convergence under generalized smoothness and directed communication graphs, even when the gradient differences among agents can be unbounded. To the best of our knowledge, this is the first time such a result has been established. To this end, we first introduce compact notation and characterize key error quantities. We then develop auxiliary results on clipped stepsizes, gradient boundedness, and error dynamics. Finally, we combine these results to prove the main convergence theorem.

\subsection{Matrix Formulation and Error Definitions}
To analyze the convergence of the proposed algorithm, we first express the update 
rules in Algorithm \ref{alg:1} in a compact matrix form. The iterations in~\eqref{Eq:3} and~\eqref{Eq:4} can 
be written as
% \begin{align}
%     \boldsymbol{x}^{k+1} &= \mathbf{R}\boldsymbol{x}^k - \boldsymbol{\alpha}_k \boldsymbol{y}^k, \\
%     \boldsymbol{y}^{k+1} &= 
%     \mathbf{C}\boldsymbol{y}^k 
%     + \nabla f(\boldsymbol{x}^{k+1}) - \nabla f(\boldsymbol{x}^k).
%     \label{Eq:x_k1}
% \end{align}
\begin{align}
    \boldsymbol{x}^{k+1} 
    &= \boldsymbol{R}\boldsymbol{x}^k - \boldsymbol{\alpha}_k \boldsymbol{y}^k, 
    \label{Eq:update_x} \\[4pt]
    \boldsymbol{y}^{k+1} 
    &= \boldsymbol{C}\boldsymbol{y}^k 
       + \nabla f(\boldsymbol{x}^{k+1}) - \nabla f(\boldsymbol{x}^k),
    \label{Eq:update_y}
\end{align}
where $\boldsymbol{x}^k = [(\boldsymbol{x}_1^k)^\top; \cdots; (\boldsymbol{x}_N^k)^\top] \in \mathbb{R}^{N \times d}$, $\boldsymbol{y}^k = [(\boldsymbol{y}_1^k)^\top; \cdots; (\boldsymbol{y}_N^k)^\top] \in \mathbb{R}^{N \times d}$, and $\boldsymbol{\alpha}^k \boldsymbol{y}^k = [({\alpha}_1^k \boldsymbol{y}_1^k)^\top; \cdots; ({\alpha}_N^k \boldsymbol{y}_N^k)^\top] \in \mathbb{R}^{N \times d}$.

We define the network-wide averaged variables as
\[
\bar{\boldsymbol{x}}^k := \tfrac{1}{N}\boldsymbol{u}^\top \boldsymbol{x}^k\in \mathbb{R}^d, 
\qquad
\bar{\boldsymbol{y}}^k := \tfrac{1}{N}\boldsymbol{1}^\top \boldsymbol{y}^k\in \mathbb{R}^d,
\]
where $\boldsymbol{u}$ is the left eigenvector of matrix $\boldsymbol{R}$ associated with 
the eigenvalue~1 (see Lemma~\ref{Lemma:uv}).  
Using the update rules in (\ref{Eq:update_x}) and (\ref{Eq:update_y}) above, we can obtain the dynamics of $\bar{\boldsymbol{x}}^{k}$ and $\bar{\boldsymbol{y}}^{k}$ as follows:
\begin{align}
    \bar{\boldsymbol{x}}^{k+1} 
    &= \bar{\boldsymbol{x}}^{k} 
       - \tfrac{1}{N}\boldsymbol{u}^\top \boldsymbol{\alpha}_k \boldsymbol{y}^k, 
       \label{Eq:dynamic_bar_x} \\
    \bar{\boldsymbol{y}}^{k+1} 
    &= \bar{\boldsymbol{y}}^{k}
       + \tfrac{1}{N}\boldsymbol{1}^\top 
       \big( \nabla f(\boldsymbol{x}^{k+1}) - \nabla f(\boldsymbol{x}^{k}) \big).
       \label{Eq:dynamic_bar_y}
\end{align}

To characterize the disagreement among agents, we define the consensus error as follows:
\begin{equation}\label{Eq:e_x^k}
    \boldsymbol{e}_{x,k} := 
    \boldsymbol{x}^{k} - \boldsymbol{1} (\bar{\boldsymbol{x}}^{k})^\top \in \mathbb{R}^{N \times d},
\end{equation}
which measures how far each agent's local variable deviates from the global
average.  

Similarly, the gradient-tracking error is defined as
\begin{equation}\label{Eq:e_y^k}
    \boldsymbol{e}_{y,k} := 
    \boldsymbol{y}^{k} - \boldsymbol{v} (\bar{\boldsymbol{y}}^{k})^\top \in \mathbb{R}^{N \times d},
\end{equation}
where $\boldsymbol{v}$ is the right eigenvector of matrix $\boldsymbol{C}$ corresponding to 
eigenvalue~1.

The $i$-th rows of the error matrices $\boldsymbol{e}_{x,k}$ and $\boldsymbol{e}_{y,k}$ satisfy
\[
\boldsymbol{e}_{x,k,i} = (\boldsymbol{x}_i^k)^\top - (\bar{\boldsymbol{x}}^{k})^\top, 
\qquad 
\boldsymbol{e}_{y,k,i} = (\boldsymbol{y}_i^k)^\top - (v_i\bar{\boldsymbol{y}}^{k})^\top.
\]

% \subsection{Error Dynamics}

Using the identities $\boldsymbol{R}\mathbf{1} = \mathbf{1}$ and 
$\mathbf{1}^\top \boldsymbol{C} = \mathbf{1}^\top$, together with the update rules
in~\eqref{Eq:3} and \eqref{Eq:4}, one can verify that the consensus error evolves as
\begin{equation}\label{Eq:e_x^k1}
    \boldsymbol{e}_{x,k+1}
    =
    \bigg( \boldsymbol{R} - \tfrac{\boldsymbol{1}\boldsymbol{u}^\top}{N} \bigg)\boldsymbol{e}_{x,k}
    -
    \bigg( \boldsymbol{I} - \tfrac{\boldsymbol{1}\boldsymbol{u}^\top}{N} \bigg)
    \boldsymbol{\alpha}_k \boldsymbol{y}^k ,
\end{equation}
and the gradient-tracking error evolves as
\begin{equation}\label{Eq:e_y^k1}
    \boldsymbol{e}_{y,k+1}
    =
    \bigg( \boldsymbol{C} - \tfrac{\boldsymbol{v}\boldsymbol{1}^\top}{N} \bigg)\boldsymbol{e}_{y,k}
    +
    \bigg( \boldsymbol{I} - \tfrac{\boldsymbol{v}\boldsymbol{1}^\top}{N} \bigg)
    \big( \nabla f(\boldsymbol{x}^{k+1}) - \nabla f(\boldsymbol{x}^k) \big).
\end{equation}

These relations are important for our convergence analysis.

\subsection{Auxiliary Results}\label{Sec:preliminary}

We first quantify how locally clipped stepsizes derivate from each other. 

\begin{lemma}\label{Lemma:pre1}
For any agent $i \in [N]$ in Algorithm~\ref{alg:1}, denote the clipped local stepsize as
\[
\alpha_i^k
= \alpha \min \left\{ 1,\; \frac{c_0}{\| \boldsymbol{y}_i^k \|} \right\},
\]
and the stepsize based on the network-average gradient as 
\[
\bar{\alpha}_i^k
= \alpha \min \left\{ 1,\; \frac{c_0}{v_i \| \nabla F(\bar{\boldsymbol{x}}^k) \|} \right\}.
\]
Then, under Assumption~\ref{Assum:lower_bound},  
Assumption~\ref{Assum:Lipschitz}, and  
Assumption~\ref{Assum:matrices_RC},
the following inequality holds:
\begin{equation}\label{Eq:pre1}
    \bigl| \alpha_i^k - \bar{\alpha}_i^k \bigr|
    \, \| \boldsymbol{y}_i^k \|
    \;\leqslant \;
    \bar{\alpha}_i^k 
    \, \| \boldsymbol{y}_i^k - v_i \nabla F(\bar{\boldsymbol{x}}^k) \|.
\end{equation}

Furthermore, denoting the global stepsize as 
\[
\bar{\alpha}_k
= \alpha \min \left\{ 1,\; \frac{c_0}{\| \boldsymbol{v} \| \| \nabla F(\bar{\boldsymbol{x}}^k) \|} \right\},
\]
then we have 
\[
\bar{\alpha}_k 
\;\leqslant \; \bar{\alpha}_i^k 
\;\leqslant \; \frac{\| \boldsymbol{v} \|}{v_i} \bar{\alpha}_k,
\]
and
\begin{equation}\label{Eq:pre1.1}
    \bigl| \alpha_i^k - \bar{\alpha}_i^k \bigr|
    \, \| \boldsymbol{y}_i^k \|
    \;\leqslant \;
    \frac{\| \boldsymbol{v} \|}{v_i}
    \bar{\alpha}_k
    \, \| \boldsymbol{y}_i^k - v_i \nabla F(\bar{\boldsymbol{x}}^k) \|.
\end{equation}
\end{lemma}
\begin{proof}
    See Appendix C.
    % ~\ref{Appendix:proof_Lemma:pre1}.
\end{proof}

To establish boundedness of the gradients, we need to derive bounds on the consensus and gradient-tracking errors in Lemma~\ref{Lemma:unibound_x} and Lemma~\ref{Lemma:unibound_y}.

\begin{lemma}\label{Lemma:unibound_x}
Suppose that Assumptions~\ref{Assum:lower_bound}, \ref{Assum:Lipschitz}, 
and~\ref{Assum:matrices_RC} hold. For the iterates generated by 
Algorithm~\ref{alg:1}, the consensus error $\|\boldsymbol{e}_{x,k}\|_R^2$ is uniformly bounded. 
Specifically, we have
\begin{equation}\label{eq:e_x_bound}
    \|\boldsymbol{e}_{x,k}\|_R^2 \leqslant \mathcal{C}_x \alpha^2 c_0^2, 
    \qquad \forall k \geqslant 0,
\end{equation}
where $\mathcal{C}_{x} = \frac{2N\sigma_R^2 (1+\sigma_R^2)\delta_{R,2}^2 \|\boldsymbol{I}-\tfrac{\boldsymbol{1}\boldsymbol{u}^{\top}}{N}\|_R^2}{(1-\sigma_R^2)^2} 
 $.
\end{lemma}
\begin{proof}
    See Appendix D.
    % ~\ref{Appendix:proof_Lemma:unibound_x}.
\end{proof}

Next, we show that
the gradient-tracking error is also uniformly bounded.

\begin{lemma}\label{Lemma:unibound_y}
Suppose that Assumptions~\ref{Assum:lower_bound}, \ref{Assum:Lipschitz}, 
\ref{Assum:gradient_dissimilarity}, and~\ref{Assum:matrices_RC} hold. 
Under Algorithm~\ref{alg:1},
if the gradient satisfies 
$\|\nabla F(\bar{\boldsymbol{x}}^k)\| \leqslant G$ for all $k \geqslant 0$,  
then the gradient-tracking error satisfies
\[
\| \boldsymbol{e}_{y,k} \|_C^{2} \leqslant \mathcal{C}_y \alpha^{2} c_0^2, 
\qquad \forall k \geqslant 0,
\]
where $\mathcal{C}_y = \frac{2(1+\sigma_{C}^{2})\delta_{C,2}^{2} \|\boldsymbol{I} - \tfrac{\boldsymbol{v}\boldsymbol{1}^{\top}}{N}\|_C^{2} }{(1-\sigma_{C}^{2})^{2}} 
\mathcal{C}_{1}$, 
with $\mathcal{C}_{1} = 2(AL_0 + BL_1 b + BL_1 \ell G)^{2} 
\bigl(2N+(1+2\sigma_R^2)\mathcal{C}_{x}\bigr)$.
\end{lemma}

\begin{proof}
    See Appendix E.
    % ~\ref{Appendix:proof_Lemma:unibound_y}.
\end{proof}

 Using these error bounds, we now establish the uniform boundedness of $\|\nabla F(\bar{\boldsymbol{x}}^{k})\|$.
\begin{lemma}\label{Lemma:Bound_trajectory}
Suppose that Assumptions~\ref{Assum:lower_bound},
\ref{Assum:Lipschitz}, \ref{Assum:gradient_dissimilarity},
and \ref{Assum:matrices_RC} hold. If $\alpha$ satisfy
$0<\alpha \leqslant \frac{\boldsymbol{u}^{\top}\boldsymbol{v}}{9LN\|\boldsymbol{v}\|^2}$,
and $c_0 = 1/\sqrt{K}$, then the iterates generated by Algorithm \ref{alg:1} satisfy
\[
\|\nabla F(\bar{\boldsymbol{x}}^{k})\|
    \leqslant G, \quad \forall\, k \leqslant K,
\]
where 
\[
    G= \sup \Bigl\{\, t > 0 \;\Bigm|\;
        t^{2} \leqslant 2\bigl(L_0 + 2L_1 t\bigr)
        \bigl(f(\bar{\boldsymbol{x}}^{0}) - \underline{f} + \alpha^3 \mathcal{C} _f \bigr)
    \Bigr\},
\] and 
\[\mathcal{C} _f = \left( \frac{3L\kappa _{uv}^2\| \boldsymbol{v} \|^2\alpha }{N}+\frac{2\kappa _{uv}^{2}\| \boldsymbol{v}\| ^2}{N\boldsymbol{u}^{\top}\boldsymbol{v}} \right) \left( 2\mathcal{C} _y+2L^2\| \boldsymbol{v}\| ^2\mathcal{C} _x \right).\]
\end{lemma}
\begin{proof}
    We prove the result by induction, building on Lemmas~\ref{Lemma:unibound_x}, \ref{Lemma:unibound_y}, and \ref{Lemma:sub-optimality gap} (Appendix B).
    % ~\ref{Appendix_Properties of $(L_0,L_1)$ functions and Useful Lemmas}).

    Clearly, for the case \(k=0\), the claim holds trivially, as $
\|\nabla F(\bar{\boldsymbol{x}}^0)\| \leqslant G
$. 

    Next, we prove that if $ \|\nabla F(\bar{\boldsymbol{x}}^k)\| \leqslant G $ holds for $k\geqslant0$, then the inequality also holds for $k+1$.  

 According to the dynamics of $\bar{\boldsymbol{x}}^k$ in (\ref{Eq:dynamic_bar_x}) and Lemma~\ref{Lemma:pre3}, we have the following inequality for $F(\bar{\boldsymbol{x}}^{k+1})$:
\begin{equation}\label{Eq:temp12}
    \begin{aligned}
        F(\bar{\boldsymbol{x}}^{k+1})\leqslant &F(\bar{\boldsymbol{x}}^k)-\left< \nabla F(\bar{\boldsymbol{x}}^k),\bar{\boldsymbol{x}}^{k+1}-\bar{\boldsymbol{x}}^k \right> \\
        & +\frac{AL_0+BL_1\| \nabla F(\bar{\boldsymbol{x}}^k)\|}{2}\| \bar{\boldsymbol{x}}^{k+1}-\bar{\boldsymbol{x}}^k\| ^2
\\
\leqslant &F(\bar{\boldsymbol{x}}^k)-\left< \nabla F(\bar{\boldsymbol{x}}^k),\frac{1}{N}\boldsymbol{u}^{\top}\boldsymbol{\alpha }_k\boldsymbol{y}^k \right> \\
&+\frac{L}{2}\| \frac{1}{N}\boldsymbol{u}^{\top}\boldsymbol{\alpha }_k\boldsymbol{y}^k\| ^2,
    \end{aligned}
\end{equation}
where $L=AL_0+BL_1 G$.

For the inner product term, we have
\begin{equation}
    \begin{aligned}\label{Eq:temp9}
&\left< \nabla F(\bar{\boldsymbol{x}}^k),\tfrac{1}{N}\boldsymbol{u}^{\top}\boldsymbol{\alpha }_k\boldsymbol{y}^k \right> \\
=& \left< \nabla F(\bar{\boldsymbol{x}}^k),\tfrac{1}{N}\boldsymbol{u}^{\top}\left( \boldsymbol{\alpha }_k\boldsymbol{y}^k-\tilde{\boldsymbol{\alpha}}_k\boldsymbol{v}\nabla F(\bar{\boldsymbol{x}}^k) \right) \right> 
\\
&+\left< \nabla F(\bar{\boldsymbol{x}}^k),\tfrac{1}{N}\boldsymbol{u}^{\top}\tilde{\boldsymbol{\alpha}}_k\boldsymbol{v}\nabla F(\bar{\boldsymbol{x}}^k) \right>, 
 \end{aligned}
\end{equation}
where 
\begin{align*}
&\tilde{\boldsymbol{\alpha}}_k\boldsymbol{v}\nabla F(\bar{\boldsymbol{x}}^k) \\
=& [(\bar{\alpha}_1^k v_1 \nabla F(\bar{\boldsymbol{x}}^k))^\top; \cdots; (\bar{\alpha}_N^k v_N \nabla F(\bar{\boldsymbol{x}}^k))^\top].
\end{align*}
We first analyze the first term on the right hand side of \eqref{Eq:temp9}, which can be verified to satisfy
% $\left< \nabla F(\bar{\boldsymbol{x}}^k),\tfrac{1}{N}\boldsymbol{u}^{\top}\left( \boldsymbol{\alpha }_k\boldsymbol{y}^k-\tilde{\boldsymbol{\alpha}}_k\boldsymbol{v}\nabla F(\bar{\boldsymbol{x}}^k) \right) \right>$. 
\begin{equation}\label{Eq:temp5}
    \begin{aligned}
 &\left< \nabla F(\bar{\boldsymbol{x}}^k),\tfrac{1}{N}\boldsymbol{u}^{\top}\left( \boldsymbol{\alpha }_k\boldsymbol{y}^k-\tilde{\boldsymbol{\alpha}}_k\boldsymbol{v}\nabla F(\bar{\boldsymbol{x}}^k) \right) \right> \\
=&\left< \nabla F(\bar{\boldsymbol{x}}^k),\tfrac{1}{N}\boldsymbol{u}^{\top}\left( \boldsymbol{\alpha }_k\boldsymbol{y}^k-\tilde{\boldsymbol{\alpha}}_k\boldsymbol{y}^k \right) \right> \\
&+\left< \nabla F(\bar{\boldsymbol{x}}^k),\tfrac{1}{N}\boldsymbol{u}^{\top}\left( \tilde{\boldsymbol{\alpha}}_k\boldsymbol{y}^k-\tilde{\boldsymbol{\alpha}}_k\boldsymbol{v}\nabla F(\bar{\boldsymbol{x}}^k) \right) \right>,
\end{aligned}
\end{equation}
where $\tilde{\boldsymbol{\alpha}}_k\boldsymbol{y}^k 
= [(\bar{\alpha}_1^k \boldsymbol{y}_1^k)^\top; \cdots; (\bar{\alpha}_N^k\boldsymbol{y}_N^k)^\top]$.

For the first term in (\ref{Eq:temp5}), by Lemma~\ref{Lemma:pre1}, we have 
\begin{equation}\label{Eq:temp6}
    \begin{aligned}
&\left| \left< \nabla F(\bar{\boldsymbol{x}}^k),\tfrac{1}{N}\boldsymbol{u}^{\top}\left( \boldsymbol{\alpha }_k\boldsymbol{y}^k-\tilde{\boldsymbol{\alpha}}_k\boldsymbol{y}^k \right) \right> \right|\\
\leqslant& \frac{1}{N}\| \nabla F(\bar{\boldsymbol{x}}^k)\| \sum_{i=1}^N{u_i\left| \alpha _{i}^{k}-\bar{\alpha}_{i}^{k} \right|}\| \boldsymbol{y}_{i}^{k}\| 
\\
\leqslant &\frac{1}{N}\| \nabla F(\bar{\boldsymbol{x}}^k)\| \sum_{i=1}^N{u_i}\bar{\alpha}_{i}^{k}\| \boldsymbol{y}_{i}^{k}-v_i\nabla F(\bar{\boldsymbol{x}}^k)\|\\
\leqslant &\frac{\| \boldsymbol{v}\|}{N}\bar{\alpha}_k\| \nabla F(\bar{\boldsymbol{x}}^k)\| \sum_{i=1}^N{\frac{u_i}{v_i}\| \boldsymbol{y}_{i}^{k}-v_i\nabla F(\bar{\boldsymbol{x}}^k)\|},
\\ 
\end{aligned}
\end{equation}
where the last inequality used the relation $\bar{\alpha}_k\leqslant \bar{\alpha}_{i}^{k}\leqslant \frac{\| \boldsymbol{v}\|}{v_i}\bar{\alpha}_k
$. Denoting $\kappa_{uv} := \sup_{i} \frac{u_i}{v_i}$, we can represent the inequality in (\ref{Eq:temp6}) as follows:
\begin{equation}\label{Eq:temp7}
    \begin{aligned}
&\left| \left< \nabla F(\bar{\boldsymbol{x}}^k),\tfrac{1}{N}\boldsymbol{u}^{\top}\left( \boldsymbol{\alpha }_k\boldsymbol{y}^k-\tilde{\boldsymbol{\alpha}}_k\boldsymbol{y}^k \right) \right> \right|\\
\leqslant& \frac{\kappa _{uv}\| \boldsymbol{v}\|}{N}\bar{\alpha}_k\| \nabla F(\bar{\boldsymbol{x}}^k)\| \sum_{i=1}^N{\| \boldsymbol{y}_{i}^{k}-v_i\nabla F(\bar{\boldsymbol{x}}^k)\|}.
\end{aligned}
\end{equation}
Similarly, for the second term on the right hand side of (\ref{Eq:temp5}), we have 
\begin{equation}\label{Eq:temp8}
    \begin{aligned}
&\left| \left< \nabla F(\bar{\boldsymbol{x}}^k),\tfrac{1}{N}\boldsymbol{u}^{\top}\left( \tilde{\boldsymbol{\alpha}}_k\boldsymbol{y}^k-\tilde{\boldsymbol{\alpha}}_k\boldsymbol{v}\nabla F(\bar{\boldsymbol{x}}^k) \right) \right> \right| \\
=&\frac{1}{N}\| \nabla F(\bar{\boldsymbol{x}}^k)\| \sum_{i=1}^N{u_i}\bar{\alpha}_{i}^{k}\| \boldsymbol{y}_{i}^{k}-v_i\nabla F(\bar{\boldsymbol{x}}^k)\| 
\\
\leqslant& \frac{\kappa _{uv}\| \boldsymbol{v}\|}{N}\bar{\alpha}_k\| \nabla F(\bar{\boldsymbol{x}}^k)\| \sum_{i=1}^N{\| \boldsymbol{y}_{i}^{k}-v_i\nabla F(\bar{\boldsymbol{x}}^k)\|} .
\end{aligned}
\end{equation}
Combining (\ref{Eq:temp7}) and (\ref{Eq:temp8}), we can bound the inner product term in (\ref{Eq:temp12}) as follows:
\begin{equation}
    \begin{aligned}\label{Eq:temp10}
&-\left< \nabla F(\bar{\boldsymbol{x}}^k),\frac{1}{N}\boldsymbol{u}^{\top}\boldsymbol{\alpha }_k\boldsymbol{y}^k \right> 
\leqslant -\frac{\boldsymbol{u}^{\top}\boldsymbol{v}}{2N}\bar{\alpha}_k\| \nabla F(\bar{\boldsymbol{x}}^k)\| ^2\\
& \qquad \qquad \qquad \qquad +\frac{2\kappa _{uv}^{2}\| \boldsymbol{v}\| ^2}{N\boldsymbol{u}^{\top}\boldsymbol{v}}\bar{\alpha}_k\| \boldsymbol{y}^k-\boldsymbol{v}\nabla F(\bar{\boldsymbol{x}}^k)\| ^2.
 \end{aligned}
\end{equation}
Leveraging the error bounds in Lemma~\ref{Lemma:unibound_x} and Lemma~\ref{Lemma:unibound_y}, we have
\begin{equation}\label{Eq:temp11}
    \begin{aligned}
       \| \boldsymbol{y}^k-\boldsymbol{v}\nabla F(\bar{\boldsymbol{x}}^k)\| ^2&=\| \boldsymbol{y}^k-\boldsymbol{v}\bar{\boldsymbol{y}}^k+\boldsymbol{v}\bar{\boldsymbol{y}}^k-\boldsymbol{v}\nabla F(\bar{\boldsymbol{x}}^k)\| ^2
\\
&\leqslant 2\left\| \boldsymbol{y}^k-\boldsymbol{v}\bar{\boldsymbol{y}}^k \right\| ^2+2\left\| \boldsymbol{v}\bar{\boldsymbol{y}}^k-\boldsymbol{v}\nabla F(\bar{\boldsymbol{x}}^k) \right\| ^2
\\
&\leqslant 2\mathcal{C} _y\alpha ^2c_0^2+2L^2\| \boldsymbol{v}\| ^2\mathcal{C} _x\alpha ^2c_0^2.
    \end{aligned}
\end{equation}

For the term $\frac{L}{2N^2}\| \boldsymbol{u}^{\top}\boldsymbol{\alpha }_k\boldsymbol{y}^k\| ^2$, we can bound it using the identity 
\(\boldsymbol{u}^{\top}\boldsymbol{\alpha}_k\boldsymbol{y}^k = \sum_{i=1}^N u_i \alpha_i^k \boldsymbol{y}_i^k\)
together with Jensen’s inequality
\(\bigl\|\sum_i z_i\bigr\|^2 \leqslant N\sum_i \|z_i\|^2\), yielding
\begin{equation}\label{Eq:quatrative_term}
\frac{L}{2N^2}\Bigl\|\sum_{i=1}^N u_i \alpha_i^k \boldsymbol{y}_i^k\Bigr\|^2
\leqslant
\frac{L}{2N}\sum_{i=1}^N u_i^2 \bigl\| \alpha_i^k \boldsymbol{y}_i^k \bigr\|^2.
\end{equation}

For $\alpha_i^k \boldsymbol{y}_i^k$ on the right hand side of \eqref{Eq:quatrative_term}, we can add and subtract \(\bar{\alpha}_i^k \boldsymbol{y}_i^k\) and $v_i\bar{\alpha}_i^k \nabla F(\bar{\boldsymbol{x}}^k)$ to yield
\begin{equation}
\alpha_i^k \boldsymbol{y}_i^k
=
(\alpha_i^k - \bar{\alpha}_i^k)\boldsymbol{y}_i^k
+ \bar{\alpha}_i^k(\boldsymbol{y}_i^k - v_i\nabla F(\bar{\boldsymbol{x}}^k))
+ v_i\bar{\alpha}_i^k \nabla F(\bar{\boldsymbol{x}}^k).
\end{equation}

Further using the inequality \(\|a+b+c\|^2 \leqslant 3(\|a\|^2+\|b\|^2+\|c\|^2)\) leads to 
\begin{equation}
    \begin{aligned}
        \|\alpha_i^k\boldsymbol{y}_i^k\|^2
\leqslant& 
3\Big( \sum_{i=1}^N{u_i^2\left\| \alpha _{i}^{k}\boldsymbol{y}_{i}^{k}-\bar{\alpha}_{i}^{k}\boldsymbol{y}_{i}^{k} \right\| ^2}\\
&+\sum_{i=1}^N{u_i^2(\bar{\alpha}_{i}^{k})^2\left\| \boldsymbol{y}_{i}^{k}-v_i\nabla F(\bar{\boldsymbol{x}}^k) \right\| ^2}\\
&+\sum_{i=1}^N{\left\| v_i\bar{\alpha}_{i}^{k}\nabla F(\bar{\boldsymbol{x}}^k) \right\| ^2} \Big).
    \end{aligned}
\end{equation}

By Lemma~\ref{Lemma:pre1},  we have:
\begin{equation}\label{Eq:ualphay}
    \begin{aligned}
&\frac{L}{2N^2}\| \boldsymbol{u}^{\top}\boldsymbol{\alpha }_k\boldsymbol{y}^k\| ^2\\
\leqslant &\frac{3L}{2N}\Big( 2\bar{\alpha}_k^2\sum_{i=1}^N{\frac{u_i^2}{v_i^2}\| \boldsymbol{v} \|^2 \left\| \boldsymbol{y}_{i}^{k}-v_i\nabla F(\bar{\boldsymbol{x}}^k) \right\| ^2}\\
&+\| \boldsymbol{v} \|^2 \bar{\alpha}_k ^2N\left\| \nabla F(\bar{\boldsymbol{x}}^k) \right\| ^2 \Big) 
\\
\leqslant &\frac{3L}{2}\| \boldsymbol{v} \|^2 \bar{\alpha}_k^2\left\| \nabla F(\bar{\boldsymbol{x}}^k) \right\| ^2\\
&+\frac{3L\| \boldsymbol{v} \|^2 \kappa _{uv}^2}{N}\bar{\alpha}_k^2 \left\| \boldsymbol{y}^k-\boldsymbol{v}\nabla F(\bar{\boldsymbol{x}}^k) \right\| ^2.
 \end{aligned}
\end{equation}

% When $\bar{\alpha}_k\leqslant \frac{\boldsymbol{u}^{\top}\boldsymbol{v}}{9NL\| \boldsymbol{v}\|}$, we have
% \begin{equation}\label{Eq:temp13}
%     \begin{aligned}
% \frac{L}{2N^2}\| \boldsymbol{u}^{\top}\boldsymbol{\alpha }_k\boldsymbol{y}^k\| ^2\leqslant\frac{\boldsymbol{u}^{\top}\boldsymbol{v}}{6N}\bar{\alpha}_k\| \nabla F(\bar{\boldsymbol{x}}^k)\| ^2+\frac{2\kappa _{uv}\boldsymbol{u}^{\top}\boldsymbol{v}}{6N^2}\| \boldsymbol{y}^k-\boldsymbol{v}\nabla F(\bar{\boldsymbol{x}}^k)\| ^2.
%  \end{aligned}
% \end{equation}   
% \begin{equation}\label{Eq:temp14}
%     \begin{aligned}
% -\left< \nabla F(\bar{\boldsymbol{x}}^k),\frac{1}{N}\boldsymbol{u}^{\top}\boldsymbol{\alpha }_k\boldsymbol{y}^k \right> \leqslant -\frac{\boldsymbol{u}^{\top}\boldsymbol{v}}{2N}\bar{\alpha}_k\| \nabla F(\bar{\boldsymbol{x}}^k)\| ^2+\frac{2\kappa _{uv}^{2}\| \boldsymbol{v}\|}{9N^2L}\| \boldsymbol{y}^k-\boldsymbol{v}\nabla F(\bar{\boldsymbol{x}}^k)\| ^2
% \end{aligned}
% \end{equation}  
Combining (\ref{Eq:temp10}), (\ref{Eq:temp11}) and (\ref{Eq:ualphay}), we obtain the following relation under $\alpha \leqslant \frac{\boldsymbol{u}^{\top}\boldsymbol{v}}{9LN\|\boldsymbol{v}\|^2}$:
\begin{equation}\label{Eq:temp15}
    \begin{aligned}
&F(\bar{\boldsymbol{x}}^{k+1})-\underline{f}+\frac{\boldsymbol{u}^{\top}\boldsymbol{v}}{3N}\bar{\alpha}_k\| \nabla F(\bar{\boldsymbol{x}}^k)\| ^2 \leqslant F(\bar{\boldsymbol{x}}^k)-\underline{f}\\
&+\left( \frac{3L\kappa _{uv}^2\| \boldsymbol{v} \|^2\bar{\alpha}_k}{N}+\frac{2\kappa _{uv}^{2}\| \boldsymbol{v}\| ^2}{N\boldsymbol{u}^{\top}\boldsymbol{v}} \right) \bar{\alpha}_k\| \boldsymbol{y}^k-\boldsymbol{v}\nabla F(\bar{\boldsymbol{x}}^k)\| ^2
\\
\leqslant & F(\bar{\boldsymbol{x}}^k)-\underline{f}+\mathcal{C} _f\alpha ^3c_{0}^{2},
\end{aligned}
\end{equation}  
where 
\begin{align}\notag
 \mathcal{C} _f &= \left( \frac{3L\kappa _{uv}^2\| \boldsymbol{v} \|^2\alpha }{N}+\frac{2\kappa _{uv}^{2}\| \boldsymbol{v}\| ^2}{N\boldsymbol{u}^{\top}\boldsymbol{v}} \right) \left( 2\mathcal{C} _y+2L^2\| \boldsymbol{v}\| ^2\mathcal{C} _x \right). \notag 
\end{align}

Taking a summation over $s\leqslant k+1\leqslant K$, under $c_0 = \frac{1}{\sqrt{K}}$, we have
\begin{equation}\label{Eq:temp16}
    \begin{aligned}
&F(\bar{\boldsymbol{x}}^{k+1})-\underline{f}+\frac{\boldsymbol{u}^{\top}\boldsymbol{v}}{3N}\sum_{s=0}^{k}{\bar{\alpha}^s\| \nabla F(\bar{\boldsymbol{x}}^s)\| ^2}\\
\leqslant &F(\bar{\boldsymbol{x}}^0)-\underline{f}+K\mathcal{C} _f\alpha ^3c_{0}^{2}\\
\leqslant &F(\bar{\boldsymbol{x}}^0)-\underline{f}+\mathcal{C} _f\alpha ^3.
\end{aligned}
\end{equation}  

Then by Lemma~\ref{Lemma:F->gradient}, we have $\|\nabla F(\bar{\boldsymbol{x}}^{k+1})\|
    \leqslant G$.

Therefore, we have $\|\nabla F(\bar{\boldsymbol{x}}^{k})\|
    \leqslant G$ for all $k \leqslant K$.
\end{proof}

With gradient boundedness established, we are now in a position to derive recursive bounds on the error dynamics.

\begin{lemma}\label{Th:consensus}
Suppose that Assumptions~\ref{Assum:lower_bound},
\ref{Assum:Lipschitz}, \ref{Assum:gradient_dissimilarity},
and \ref{Assum:matrices_RC} hold. If the stepsize
$\alpha$ satisfies
\begin{equation}\label{Eq:stepsize_bound_1}
\begin{aligned}
0 < \alpha \;\leqslant\;
\min\Biggl\{~
&\tfrac{(1-\sigma_R^{2})\sqrt{N}}
{6\sqrt{2} \|\boldsymbol{v}\|\kappa _v\|\boldsymbol{I}-\tfrac{\boldsymbol{1}\boldsymbol{u}^{\top}}{N}\|_R\delta_{R,2}
\left(AL_0+BL_1b+BL_1\ell G\right)},
\\[3pt]
&\tfrac{1-\sigma _{C}^{2}}{12\sqrt{2} \kappa_v \delta_{C,2}\| \boldsymbol{I}-\frac{v\mathbf{1}^{\top}}{N} \| _{C}(AL_0+BL_1 b+BL_1\ell G )}
~\Biggr\},
\end{aligned}
\end{equation}
then, the consensus error $\boldsymbol{e}_{x,k}$ and the gradient-tracking
error $\boldsymbol{e}_{y,k}$ satisfy the following relations for all
$k \geqslant 0$:

\begin{equation}\label{Equ:consensus_errors}
\begin{aligned}
\| \boldsymbol{e}_{x,k+1}\| _{R}^{2}\leqslant &\tfrac{1+\sigma _{R}^{2}}{2}\| \boldsymbol{e}_{x,k}\| _{R}^{2}+\alpha^2\mathcal{C} _{x,1}\| \boldsymbol{e}_{y,k}\| _{C}^{2}\\
&+ \alpha\, \mathcal{C} _{x,2} \, \bar{\alpha}_k\| \nabla F(\bar{\boldsymbol{x}}^k)\| ^2,
\\
\| \boldsymbol{e}_{y,k+1}\| _{C}^{2}\leqslant &\mathcal{C} _{y,1}\| \boldsymbol{e}_{x,k}\| _{R}^{2}+\tfrac{1+\sigma _{C}^{2}}{2}\| \boldsymbol{e}_{y,k}\| _{C}^{2}\\
&+ \alpha\, \mathcal{C} _{y,2}\bar{\alpha}_k\| \nabla F(\bar{\boldsymbol{x}}^k)\| ^2,
\end{aligned}
\end{equation}
where 
\begin{equation}\label{Eq:C_x1}
\mathcal{C}_{x,1}
=
\tfrac{12(1+2\sigma_R^{2})
\|\boldsymbol{I}-\tfrac{\boldsymbol{1}\boldsymbol{u}^{\top}}{N}\|_R^{2}\delta_{R,2}^{2}\kappa_v^2}
{1-\sigma_R^{2}},
\end{equation}

\begin{equation}\label{Eq:C_x2}
\mathcal{C}_{x,2}
=
\tfrac{3N\left( 1+\sigma _R^{2} \right) \| \boldsymbol{I}-\tfrac{\boldsymbol{1}\boldsymbol{u}^{\top}}{N}\| _R^{2}\delta _{R,2}^{2}\| \boldsymbol{v}\| ^2}{1-\sigma _R^{2}},
\end{equation}

\begin{equation}\label{Eq:C_y1}
\mathcal{C}_{y,1}
=
\tfrac{(1+2\sigma_{C}^{2})
\|\boldsymbol{I}-\tfrac{\boldsymbol{v}\mathbf{1}^{\top}}{N}\|_C^{2}\delta_{C,2}^2}
{1-\sigma_{C}^{2}}(2\sigma_R^2+\tfrac{(1+2\sigma_{C}^{2})(1-\sigma_{C}^{2})\|\boldsymbol{v}\|^2}{12N}),
\end{equation}

\begin{equation}\label{Eq:C_y2}
\mathcal{C}_{y,2}
=
\tfrac{12N(1+\sigma_{C}^{2})
\bigl\|\boldsymbol{I}-\tfrac{\boldsymbol{v}\mathbf{1}^{\top}}{N}\bigr\|_C^{2}
\|\boldsymbol{v}\|^{2}\kappa^2_v \delta_{C,2}^2
\bigl(AL_0+BL_1b+BL_1\ell G\bigr)^{2}
}
{1-\sigma_{C}^{2}}.
\end{equation}

\end{lemma}
\begin{proof}
    See Appendix F.
    % ~\ref{Appendix:proof_Th:consensus}.
\end{proof}

Having characterized the consensus and gradient-tracking errors, we now characterize the descent
behavior of $\nabla F(\bar{\boldsymbol{x}}^k)$. This descent property will serve as the key ingredient
in establishing the convergence guarantee of Algorithm~\ref{alg:1}.

\begin{lemma}\label{Th:Descent}
Suppose that Assumptions~\ref{Assum:Lipschitz}, \ref{Assum:gradient_dissimilarity}, 
and~\ref{Assum:matrices_RC} hold. If the stepsize satisfies 
$\alpha \leqslant \frac{\boldsymbol{u}^{\top}\boldsymbol{v}}{9LN\|\boldsymbol{v}\|^2}$, 
then the following inequality holds for the iterates generated by Algorithm\ref{alg:1}:
\begin{equation}\label{Equ:Descent}
\begin{aligned}
&\frac{\boldsymbol{u}^{\top}\boldsymbol{v}}{3N}\bar{\alpha}_k\| \nabla F(\bar{\boldsymbol{x}}^k)\| ^2\leqslant F(\bar{\boldsymbol{x}}^k)-F(\bar{\boldsymbol{x}}^{k+1})
\\
& \qquad +( \frac{6L\kappa _{uv}\| \boldsymbol{v} \|}{N}+\frac{4\kappa _{uv}^{2}\| \boldsymbol{v}\| ^2}{N\boldsymbol{u}^{\top}\boldsymbol{v}} ) \bar{\alpha}_k\| \boldsymbol{e}_{y,k}\|_C ^2
\\
& \qquad +( \frac{6L\kappa _{uv}\| \boldsymbol{v} \| ^3}{N^2}+\frac{4\kappa _{uv}^{2}\| \boldsymbol{v}\| ^4}{N^2\boldsymbol{u}^{\top}\boldsymbol{v}} )\times \\
& \quad \qquad 
\left( AL_0+BL_1(b+\ell G) \right) ^2\bar{\alpha}_k\| \boldsymbol{e}_{x,k}\|_R ^2.
\end{aligned}
\end{equation} 
\end{lemma}
\begin{proof}
    See Appendix G.
    % ~\ref{Appendix:proof_Th:Descent}.
\end{proof}

Lemma~\ref{Th:Descent} establishes a descent inequality for the
gradient norm of the average optimization variable. Together with the recursive error bounds derived in
Lemma~\ref{Th:consensus}, this result enables us to characterize the accumulated consensus errors and gradient-tracking
errors over time. We are now in a position to establish the convergence of Algorithm~\ref{alg:1}.

\subsection{Convergence Results}      

\begin{lemma}\label{Th:convergence}
Suppose that the conditions in Lemma~\ref{Th:consensus} and
Lemma~\ref{Th:Descent} hold. If the stepsize additionally satisfies
% \[
% 0 < \alpha \leqslant 
% \sqrt{\tfrac{(1-\sigma_{R}^{2})(1-\sigma_{C}^{2})}
%       {8\,\mathcal{C}_{x,1}\mathcal{C}_{y,1}}},
% \]
\begin{equation}
\begin{aligned}
0 < \alpha \;&\leqslant\;
\min\left\{~
      \sqrt{\tfrac{(1-\sigma_{R}^{2})(1-\sigma_{C}^{2})}
      {8\,\mathcal{C}_{x,1}\mathcal{C}_{y,1}}},\right.
\\[3pt]
&\left.\sqrt{\tfrac{N (\boldsymbol{u}^{\top}\boldsymbol{v})^2(1-\sigma_R^{2})}
{\mathcal{C}_{x,2}\|\boldsymbol{v}\|^3(144L\kappa_{uv}\boldsymbol{u}^{\top}\boldsymbol{v}+96\kappa_{uv}^{2}\|\boldsymbol{v}\|)
(AL_0+BL_1(b+\ell G))^2}}
~\right\},
\end{aligned}
\end{equation}
then, for all $K\geqslant 0$, the following results hold:
\begin{equation}
\begin{aligned}
\tfrac{1}{K}\sum_{k=0}^{K}\| \boldsymbol{e}_{x,k} \|_{R}^{2}
&\leqslant \mathcal{O}\!\left(\tfrac{N}{K(1-\sigma^2_R)}\right), \\[4pt]
\tfrac{1}{K}\sum_{k=0}^{K}\| \boldsymbol{e}_{y,k} \|_{C}^{2}
&\leqslant \mathcal{O}\!\left(\tfrac{1}{K(1-\sigma^2_R)(1-\sigma^2_C)}\right), \\[4pt]
\tfrac{1}{K}\sum_{k=0}^{K-1}
\bar{\alpha}_{k}\, \| \nabla F(\bar{\boldsymbol{x}}^{k}) \|^{2}
&\leqslant \mathcal{O}\!\left(\tfrac{N}{K\boldsymbol{u}^{\top}\boldsymbol{v}}\right).
\end{aligned}
\end{equation}
\end{lemma}

\begin{proof}
By Lemma~\ref{Th:consensus},
$\| \boldsymbol{e}_{x,k+1}\| _{R}^{2}$ and
$\| \boldsymbol{e}_{y,k+1}\| _{C}^{2}$ satisfy the following system of inequalities:
\begin{equation}\label{Eq:dynamic_sys}
\begin{aligned}
\begin{bmatrix}
    \| \boldsymbol{e}_{x,k+1}\| _{R}^{2} \\
    \| \boldsymbol{e}_{y,k+1}\| _{C}^{2}
\end{bmatrix}
\leqslant&
\begin{bmatrix}
   \tfrac{1+\sigma _{R}^{2}}{2} & \alpha ^2\mathcal{C} _{x,1} \\[6pt]
    \mathcal{C} _{y,1} & \tfrac{1+\sigma _{C}^{2}}{2}
\end{bmatrix}
\begin{bmatrix}
    \| \boldsymbol{e}_{x,k}\| _{R}^{2} \\
    \| \boldsymbol{e}_{y,k}\| _{C}^{2}
\end{bmatrix} \\
& +
\begin{bmatrix}
    \alpha \,\mathcal{C} _{x,2} \\[4pt]
    \alpha \,\mathcal{C} _{y,2}
\end{bmatrix}
\bar{\alpha}_k\| \nabla F(\bar{\boldsymbol{x}}^k)\| ^2.
\end{aligned}
\end{equation}

Define the stacked error vector as
\[
\boldsymbol{u}_{k}=
\begin{bmatrix}
    \| \boldsymbol{e}_{x,k}\| _{R}^{2} \\
    \| \boldsymbol{e}_{y,k}\| _{C}^{2}
\end{bmatrix},
\]
the system matrix as
\[
\boldsymbol{G}=
\begin{bmatrix}
   \tfrac{1+\sigma _{R}^{2}}{2} & \alpha ^2\mathcal{C} _{x,1} \\[6pt]
    \mathcal{C} _{y,1} & \tfrac{1+\sigma _{C}^{2}}{2}
\end{bmatrix},
\]
and the input term as
\[
\boldsymbol{b}_k=
\begin{bmatrix}
    \alpha \mathcal{C} _{x,2} \\[4pt]
    \alpha \mathcal{C} _{y,2}
\end{bmatrix}
\bar{\alpha}_k
\| \nabla F(\bar{\boldsymbol{x}}^k)\| ^2,
\]
then the inequality in \eqref{Eq:dynamic_sys} can be written compactly as
\begin{equation}\label{Equ:LTI1}
    \boldsymbol{u}_{k+1} \leqslant \boldsymbol{G}\boldsymbol{u}_k + \boldsymbol{b}_k .
\end{equation}

Under the stepsize condition
\(
\alpha\leqslant 
\sqrt{\tfrac{(1-\sigma _{R}^{2})(1-\sigma _{C}^{2})}
           {8\mathcal{C} _{x,1}\mathcal{C} _{y,1}}},
\)
the matrix $(I_2-\boldsymbol{G})$ is invertible since $|I_2-G|\geqslant \tfrac{\left( 1-\sigma _{R}^{2} \right) \left( 1-\sigma _{C}^{2} \right)}{8}$. Furthermore, we have the following relationship based on matrix inversion:
\[
(I_2-\boldsymbol{G})^{-1} \leqslant 
\begin{bmatrix}
	\dfrac{4}{1-\sigma _{R}^{2}} &
	\dfrac{8\alpha ^2\mathcal{C} _{x,1}}
	{(1-\sigma _{R}^{2})(1-\sigma _{C}^{2})}
\\[8pt]
	\dfrac{8\mathcal{C} _{y,1}}
	{(1-\sigma _{R}^{2})(1-\sigma _{C}^{2})} &
	\dfrac{4}{1-\sigma _{C}^{2}}
\end{bmatrix}.
\]

Therefore, recursively applying \eqref{Equ:LTI1} from $k=0$ to $K$ gives
\begin{equation}\label{Equ:LTI2}
    \sum_{k=0}^{K} \boldsymbol{u}_k
    \leqslant (I_2-\boldsymbol{G})^{-1}\boldsymbol{u}_0
    + (I_2-\boldsymbol{G})^{-1}\sum_{k=0}^{K-1}\boldsymbol{b}_k .
\end{equation}

In light of equation~(\ref{Equ:LTI2}), we further compute an entry-wise upper bound on the consensus error:
\begin{equation}\label{Equ:LTI3}
\begin{aligned}
&\sum_{k=0}^{K}{\| \boldsymbol{e}_{x,k}}\| _{R}^{2}
\leqslant \frac{8\alpha ^2\mathcal{C} _{x,1}}{\left( 1-\sigma _{R}^{2} \right) \left( 1-\sigma _{C}^{2} \right)}\!\| \boldsymbol{e}_{y,0}\| _{C}^{2}\\
& \qquad +\left( \frac{4\alpha \,\mathcal{C} _{x,2}}{1-\sigma _{R}^{2}}+\frac{8\alpha ^2\mathcal{C} _{x,1}}{\left( 1-\sigma _{R}^{2} \right) \left( 1-\sigma _{C}^{2} \right)} \right) \sum_{k=0}^{K-1}{\bar{\alpha}_k\| \nabla F(\bar{\boldsymbol{x}}^k)\| ^2}
,
\end{aligned}
\end{equation}
\begin{equation}\label{Equ:LTI4}
\begin{aligned}
&\sum_{k=0}^K{\|\boldsymbol{e}_{y,k}\|_C ^2} 
 \leqslant \frac{4}{1-\sigma _{C}^{2}}\| \boldsymbol{e}_{y,0}\| _{C}^{2}\\
& \qquad +\left[ \frac{8\alpha \,\mathcal{C} _{x,2}\mathcal{C} _{y,1}}{\left( 1-\sigma _{R}^{2} \right) \left( 1-\sigma _{C}^{2} \right)}+\frac{4\alpha \,\mathcal{C} _{y,2}}{1-\sigma _{C}^{2}} \right] \sum_{k=0}^{K-1}{\bar{\alpha}_k\| \nabla F(\bar{\boldsymbol{x}}^k)\| ^2} .
\end{aligned}
\end{equation}

   Finally, combining \eqref{Equ:Descent}, \eqref{Equ:LTI3}, and \eqref{Equ:LTI4},  we obtain the following results under  $\alpha \leqslant \sqrt{\tfrac{\boldsymbol{u}^{\top}\boldsymbol{v}\left( 1-\sigma _{R}^{2} \right)}{(\frac{144L\kappa _{uv}\| \boldsymbol{v} \| ^3}{N}+\frac{96\kappa _{uv}^{2}\| \boldsymbol{v}\| ^4}{N\boldsymbol{u}^{\top}\boldsymbol{v}})\mathcal{C} _{x,2}\left( AL_0+BL_1\left( b+\ell G \right) \right) ^2}}
$:
\begin{equation}\label{Eq:Th4.1}
\begin{aligned}
\sum_{k=0}^K{ \| \boldsymbol{e}_{x,k}\| _{R}^{2} }\leqslant& \mathcal{R} _{x,1}\bigl( F(\bar{\boldsymbol{x}}^0)-\underline{f} \bigr)
+\mathcal{R} _{x,2} \| \boldsymbol{e}_{y,0}\| _{C}^{2},
 \\[1ex]
\sum_{k=0}^K{ \| \boldsymbol{e}_{y,k}\| _{C}^{2}}\leqslant& \mathcal{R} _{y,1}\bigl( F(\bar{\boldsymbol{x}}^0)-\underline{f} \bigr)+\mathcal{R} _{y,2}\| \boldsymbol{e}_{y,0}\| _{C}^{2} ,
\\[1ex]
\sum_{k=0}^{K-1}{\bar{\alpha}_k\,  \| \nabla F(\bar{\boldsymbol{x}}^k)\| ^2}\leqslant& \tfrac{6N}{\boldsymbol{u}^{\top}\boldsymbol{v}}\bigl( F(\bar{\boldsymbol{x}}^0)-\underline{f} \bigr)+\mathcal{R} _{\nabla ,1} \| \boldsymbol{e}_{y,0}\| _{C}^{2}
,
\end{aligned}
\end{equation}
where the constants are given by
\begin{equation}\notag
\begin{aligned}
\mathcal{R}_{\nabla,1} =&
\tfrac{6\kappa_{uv}\|\boldsymbol{v}\|}{\boldsymbol{u}^{\top}\boldsymbol{v}}
\Big[
(6L+\tfrac{4\kappa_{uv}\|\boldsymbol{v}\|}{\boldsymbol{u}^{\top}\boldsymbol{v}})\tfrac{4\alpha}{1-\sigma_{C}^{2}}\\
&+
(6L\|\boldsymbol{v}\|^{2}+\tfrac{4\kappa_{uv}\|\boldsymbol{v}\|^{3}}{\boldsymbol{u}^{\top}\boldsymbol{v}})
\tfrac{8\alpha^{5}\mathcal{C}_{x,1}L^{2}}{(1-\sigma_{R}^{2})(1-\sigma_{C}^{2})}
\Big]\\
=& \mathcal{O}(\tfrac{1}{N}),
\end{aligned}
\end{equation}
\begin{equation}\notag
\begin{aligned}
\mathcal{R}_{x,1} &=
\tfrac{24N\mathcal{C}_{x,2}}{\boldsymbol{u}^{\top}\boldsymbol{v}(1-\sigma_{R}^{2})}\alpha
= \mathcal{O}\!\left(\tfrac{N}{1-\sigma_{R}^{2}}\right),
\\[4pt]
\mathcal{R}_{x,2} &=
\tfrac{4\mathcal{R}_{\nabla,1}\mathcal{C}_{x,2}}{1-\sigma_{R}^{2}}\alpha
+
\tfrac{8\mathcal{C}_{x,1}}{(1-\sigma_{R}^{2})(1-\sigma_{C}^{2})}\alpha^{2}
= \mathcal{O}\!\left(\tfrac{1}{N^2}\right),
\end{aligned}
\end{equation}
\begin{equation}\notag
\begin{aligned}
\mathcal{R}_{y,1}
&= 
\tfrac{6N}{u^\top v}\!\left(
\tfrac{8C_{x,2}C_{y,1}}{(1-\sigma_R^2)(1-\sigma_C^2)}
+\tfrac{4C_{y,2}}{1-\sigma_C^2}
\right)\!\alpha
\\
&= \mathcal{O}\!\left(\tfrac{1}{(1-\sigma_R^2)(1-\sigma_C^2)}\right),
\\[4pt]
\mathcal{R}_{y,2} &=
\mathcal{R}_{\nabla,1}
\left(
\tfrac{8\mathcal{C}_{x,2}\mathcal{C}_{y,1}}{(1-\sigma_{R}^{2})(1-\sigma_{C}^{2})}
+
\tfrac{4\mathcal{C}_{y,2}}{1-\sigma_{C}^{2}}
\right)\alpha
+
\tfrac{4}{1-\sigma_{C}^{2}}
\\
&=
\mathcal{O}\!\left(\tfrac{1}{1-\sigma_{C}^{2}}\right).
\end{aligned}
\end{equation}

\end{proof}

We are now ready to present the main convergence result, which
guarantees convergence to an $\epsilon$-stationary point.

\begin{theorem}\label{Th:epsilon_convergence}
Let Assumptions~\ref{Assum:lower_bound}, \ref{Assum:Lipschitz}, \ref{Assum:gradient_dissimilarity},
and \ref{Assum:matrices_RC} hold.  Then under 

% $0<\alpha \leqslant \min\{C_1,\,C_2,\,C_3,\,C_4,\,C_5\} \footnote{\scriptsize
% $C_1=\tfrac{(1-\sigma_R^{2})\sqrt{N}}
% {6\sqrt{2} \|\boldsymbol{v}\|\kappa _v\|\boldsymbol{I}-\tfrac{\boldsymbol{1}\boldsymbol{u}^{\top}}{N}\|_R\delta_{R,2}
% \left(AL_0+BL_1b+BL_1\ell G\right)}$,
% $C_2=\tfrac{1-\sigma _{C}^{2}}{12\sqrt{2} \kappa_v \delta_{C,2}\| \boldsymbol{I}-\frac{v\mathbf{1}^{\top}}{N} \| _{C}(AL_0+BL_1 b+BL_1\ell G )}$,
% $C_3=\sqrt{\tfrac{(1-\sigma_R^{2})(1-\sigma_C^{2})}
% {8\mathcal{C}_{x,1}\mathcal{C}_{y,1}}}$, $C_4=\tfrac{\boldsymbol{u}^{\top}\boldsymbol{v}}{9LN\|\boldsymbol{v}\|^2}$, and $C_5=\sqrt{\tfrac{\boldsymbol{u}^{\top}\boldsymbol{v}\left( 1-\sigma _{R}^{2} \right)}{(\frac{144L\kappa _{uv}\| \boldsymbol{v} \| ^3}{N}+\frac{96\kappa _{uv}^{2}\| \boldsymbol{v}\| ^4}{N\boldsymbol{u}^{\top}\boldsymbol{v}})\mathcal{C} _{x,2}\left( AL_0+BL_1\left( b+\ell G \right) \right) ^2}}$.}$
% and 
clipping threshold~$c_0=\tfrac{1}{\sqrt{K}}$, for any $\epsilon>0$, there exists $k^\star \in \{0, 1, \ldots, K-1\}$ such that the iterates generated by Algorithm~\ref{alg:1} satisfy:
\begin{enumerate}
    \item $\|\nabla F(\bar{\boldsymbol{x}}^{k^\star})\| \leqslant \epsilon$,
    \item $\max_{1 \leqslant i \leqslant N} \|\boldsymbol{x}_i^{k^\star} - \bar{\boldsymbol{x}}^{k^\star}\|_R^2 \leqslant \epsilon$,
\end{enumerate}
after 
\[
K = \mathcal{O}\!\left(
\frac{1}{\alpha^2\epsilon^2}
\right)
\]
iterations, when the stepsize $\alpha$ satisfies $0 < \alpha \leqslant \min\{C_1, C_2, C_3, C_4, C_5\}$ with
\begin{align}
C_1 &= \tfrac{(1-\sigma_R^{2})\sqrt{N}}
{6\sqrt{2} \|\boldsymbol{v}\|\kappa_v\|\boldsymbol{I}-\tfrac{\boldsymbol{1}\boldsymbol{u}^{\top}}{N}\|_R\delta_{R,2}
(AL_0+BL_1 b+BL_1\ell G)}, \\[6pt]
C_2 &= \tfrac{1-\sigma_C^{2}}
{12\sqrt{2} \kappa_v \delta_{C,2}\|\boldsymbol{I}-\tfrac{\boldsymbol{v}\mathbf{1}^{\top}}{N}\|_C(AL_0+BL_1 b+BL_1\ell G)}, \\[6pt]
C_3 &= \sqrt{\tfrac{(1-\sigma_R^{2})(1-\sigma_C^{2})}{8\mathcal{C}_{x,1}\mathcal{C}_{y,1}}}, \quad
C_4 = \tfrac{\boldsymbol{u}^{\top}\boldsymbol{v}}{9LN\|\boldsymbol{v}\|^2}, \\[6pt]
C_5 &= \sqrt{\tfrac{N (\boldsymbol{u}^{\top}\boldsymbol{v})^2(1-\sigma_R^{2})}
{\mathcal{C}_{x,2}\|\boldsymbol{v}\|^3(144L\kappa_{uv}\boldsymbol{u}^{\top}\boldsymbol{v}+96\kappa_{uv}^{2}\|\boldsymbol{v}\|)
(AL_0+BL_1(b+\ell G))^2}},
\end{align}
where the constants $\mathcal{C}_{x,1}$, $\mathcal{C}_{x,2}$ and $\mathcal{C}_{y,1}$ are defined in \eqref{Eq:C_x1}, \eqref{Eq:C_x2} and \eqref{Eq:C_y1}.
\end{theorem}
\begin{proof}
From Lemma~\ref{Th:convergence}, we have
\begin{equation}\label{eq:grad_avg_bound}
\begin{aligned}
\sum_{k=0}^{K-1}\bar{\alpha}_k\|\nabla F(\bar{\boldsymbol{x}}^k)\|^{2}
\leqslant&
 \tfrac{6N}{\boldsymbol{u}^{\top}\boldsymbol{v}}\bigl( F(\bar{\boldsymbol{x}}^0)-\underline{f} \bigr)
+\mathcal{R} _{\nabla ,1} \| \boldsymbol{e}_{y,0}\| _{C}^{2}.
\end{aligned}
\end{equation}

We substitute the stepsize
$
\bar{\alpha}_k = \alpha \min\!\left\{1,\;
\tfrac{c_0}{\|\boldsymbol{v}\|\|\nabla F(\bar{\boldsymbol{x}}^k)\|}\right\}
$
into~\eqref{eq:grad_avg_bound} and divide the iterations into two sets according to the gradient magnitude:
\[
\mathcal{S} =
\left\{0\leqslant k \leqslant  K-1 \,\middle|\,
\|\boldsymbol{v}\|\|\nabla F(\bar{\boldsymbol{x}}^k)\| < c_0 \right\},
\]
and
\[
\mathcal{S}^{C} = \left\{0\leqslant k \leqslant  K-1 \,\middle|\,
\|\boldsymbol{v}\|\|\nabla F(\bar{\boldsymbol{x}}^k)\| \geqslant c_0 \right\}.
\]
Accordingly, the inequality \eqref{eq:grad_avg_bound} can be rewritten as
\begin{equation}\label{eq:set_bounds}
\begin{aligned}
\sum_{k\in\mathcal{S}}\|\nabla F(\bar{\boldsymbol{x}}^k)\|^{2}
&\leqslant
\mathcal{O}\left(\frac{1}{\alpha}\right),
\\[2pt]
\sum_{k\in\mathcal{S}^{C}}\|\nabla F(\bar{\boldsymbol{x}}^k)\|
&\leqslant
\mathcal{O}\left(\frac{1}{\alpha c_0}\right).
\end{aligned}
\end{equation}

Next, using the Cauchy--Schwarz inequality, we have
\[
\Bigl(\sum_{k\in\mathcal{S}}\|\nabla F(\bar{\boldsymbol{x}}^k)\|\Bigr)^2
\leqslant
|\mathcal{S}|
\sum_{k\in\mathcal{S}}\|\nabla F(\bar{\boldsymbol{x}}^k)\|^{2},
\]
which further leads to the following inequality based on \eqref{eq:set_bounds}
\begin{equation}\label{eq:S_bound}
\sum_{k\in\mathcal{S}}\|\nabla F(\bar{\boldsymbol{x}}^k)\|
\leqslant
\mathcal{O}\left( \sqrt{\frac{
|\mathcal{S}|}
{\alpha}} \right).
\end{equation}

Combining~\eqref{eq:set_bounds} and~\eqref{eq:S_bound}, we can bound the average gradient norm over all iterations as
\begin{equation}\label{eq:final_avg}
\begin{aligned}
\frac{1}{K}\sum_{k=0}^{K-1}\|\nabla F(\bar{\boldsymbol{x}}^k)\|
\leqslant~&
\frac{1}{K}\!\left(
\sum_{k\in\mathcal{S}}\!\|\nabla F(\bar{\boldsymbol{x}}^k)\|
+\sum_{k\in\mathcal{S}^{C}}\!\|\nabla F(\bar{\boldsymbol{x}}^k)\|
\right)
\\[2pt]
\leqslant~&
 \mathcal{O} \left(\frac{\sqrt{\tfrac{|\mathcal{S}|}
{\alpha}}
+\tfrac{1}
{\alpha c_0}}{K}
\right).
\end{aligned}
\end{equation}

Since $|\mathcal{S}|\leqslant K$ and $c_0=\tfrac{1}{\sqrt{K}}$, inequality~\eqref{eq:final_avg} simplifies to
\begin{equation}\label{Eq:average_grad_avgxk}
\frac{1}{K}\sum_{k=0}^{K-1}\|\nabla F(\bar{\boldsymbol{x}}^k)\|
\leqslant
 \mathcal{O} \left(\frac{\sqrt{\tfrac{1}
{\alpha}}
+\tfrac{1}
{\alpha }}{\sqrt{K}}
\right).
\end{equation}
 Moreover, from Lemma~\ref{Th:convergence}, we have 
\begin{equation}\label{Eq:average_consensus_xk}
\begin{aligned}
\frac{1}{K}\sum_{k=0}^{K-1}\| \boldsymbol{e}_{x,k} \|_{R}^{2}
&\leqslant \mathcal{O}\!\left(\frac{1}{K}\right).
\end{aligned}
\end{equation}

Combining \eqref{Eq:average_grad_avgxk} and \eqref{Eq:average_consensus_xk}, we obtain
\begin{equation}\label{Eq:average_sum}
\frac{1}{K}\sum_{k=0}^{K-1}\left(\|\nabla F(\bar{\boldsymbol{x}}^k)\|+  \| \boldsymbol{e}_{x,k} \|_{R}^{2}        \right)
\leqslant
 \mathcal{O} \left(\frac{1}{\sqrt{\alpha K}}+\frac{1}{K}
\right).
\end{equation}

Therefore, for a sufficiently large $K = \mathcal{O}\left(\frac{1}{\alpha^2\epsilon^2}\right)$, 
inequality~\eqref{Eq:average_sum} implies
\begin{equation}\label{Eq:norms_leq_epsilon}
\min_{0\leqslant k \leqslant K-1} \left(\|\nabla F(\bar{\boldsymbol{x}}^k)\| + \|\boldsymbol{e}_{x,k}\|_{R}^{2} \right)
\leqslant \epsilon.
\end{equation}
Since both terms on the left of \eqref{Eq:norms_leq_epsilon} are nonnegative, there exists an iteration $k^\star \in \{0, 1, \ldots, K-1\}$ such that
\begin{equation}\label{Eq:joint_bound}
\|\nabla F(\bar{\boldsymbol{x}}^{k^\star})\| \leqslant \epsilon
\quad \text{and} \quad 
\|\boldsymbol{e}_{x,k^\star}\|_{R}^{2} \leqslant \epsilon.
\end{equation}
The second inequality in~\eqref{Eq:joint_bound}, together with the definition 
$\|\boldsymbol{e}_{x,k}\|_R^2 = \sum_{i=1}^{N}\|\boldsymbol{x}_i^k - \bar{\boldsymbol{x}}^k\|_R^2$, implies
\begin{equation}
\max_{1 \leqslant i \leqslant N} \|\boldsymbol{x}_i^{k^\star} - \bar{\boldsymbol{x}}^{k^\star}\|_R^2 
\leqslant \epsilon.
\end{equation}      

\end{proof}

\begin{remark}
  Theorem~\ref{Th:epsilon_convergence} establishes that Algorithm~\ref{alg:1} can achieve an optimization error of $\min_{0\leqslant k \leqslant {K-1}}\|\nabla F(\bar{\boldsymbol{x}}^k)\|
\leqslant \epsilon$ in $\mathcal{O}(1/\epsilon^2)$ iterations while simultaneously maintaining consensus among agents. This matches existing results for centralized optimization under ($L_0, L_1$)-smoothness in~\cite{zhang2019gradient}.
\end{remark}  

\section{Numerical Experiments}\label{Sec:experiments}

In this section, we evaluate the effectiveness of our proposed algorithm through experiments on benchmark datasets using regularized logistic regression and a convolutional neural network (CNN).  The two experiments were performed under communication matrices   
$\boldsymbol{R}$ and $\boldsymbol{C}$ depicted in Fig.~\ref{Fig:R} and Fig.~\ref{Fig:C}.

% \begin{figure}[t]
% \centering
% \begin{minipage}{0.48\textwidth}
%     \centering
%     \includegraphics[width=0.68\linewidth]{R.eps}  
%     \vspace{-3pt}
%     \caption{Directed communication graph associated with the row-stochastic mixing matrix $\boldsymbol{R}$.}
%     \label{Fig:R}
% \end{minipage}\hfill
% \begin{minipage}{0.48\textwidth}
%     \centering
%     \includegraphics[width=0.68\linewidth]{C.eps}             
%     \vspace{-3pt}
%     \caption{Directed communication graph associated with the column-stochastic mixing matrix $\boldsymbol{C}$.}
%     \label{Fig:C}
% \end{minipage}
% \vspace{-6pt}
% \end{figure}

\begin{figure}[t]
\centering
\begin{subfigure}[t]{0.45\linewidth}
    \centering
    \includegraphics[width=\linewidth]{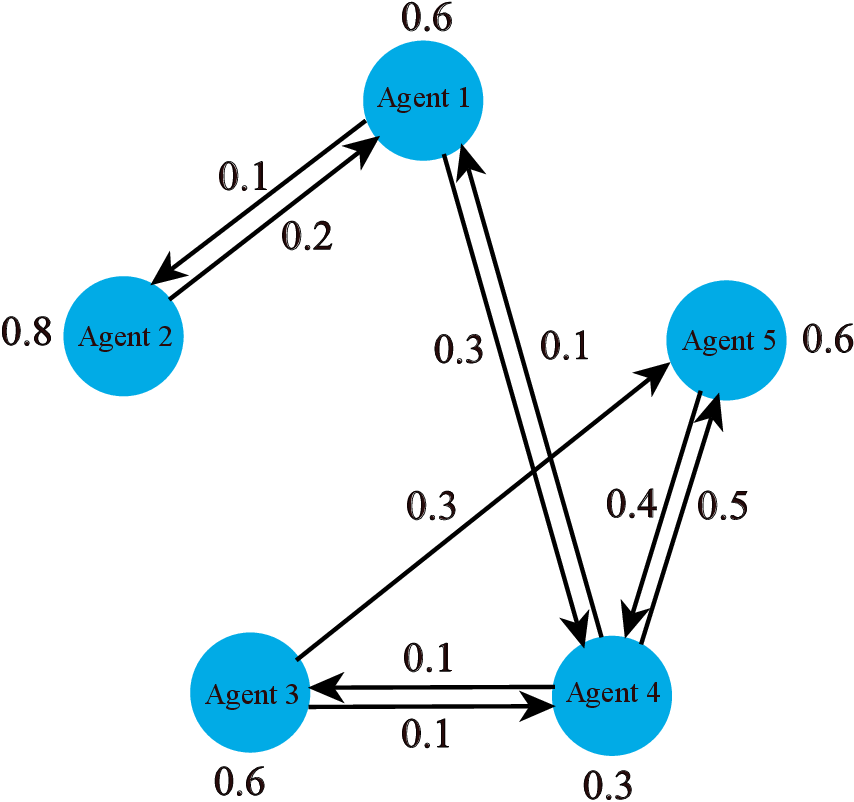}
    \caption{Row-stochastic mixing matrix $\boldsymbol{R}$.}
    \label{Fig:R}
\end{subfigure}\hfill
\begin{subfigure}[t]{0.45\linewidth}
    \centering
    \includegraphics[width=\linewidth]{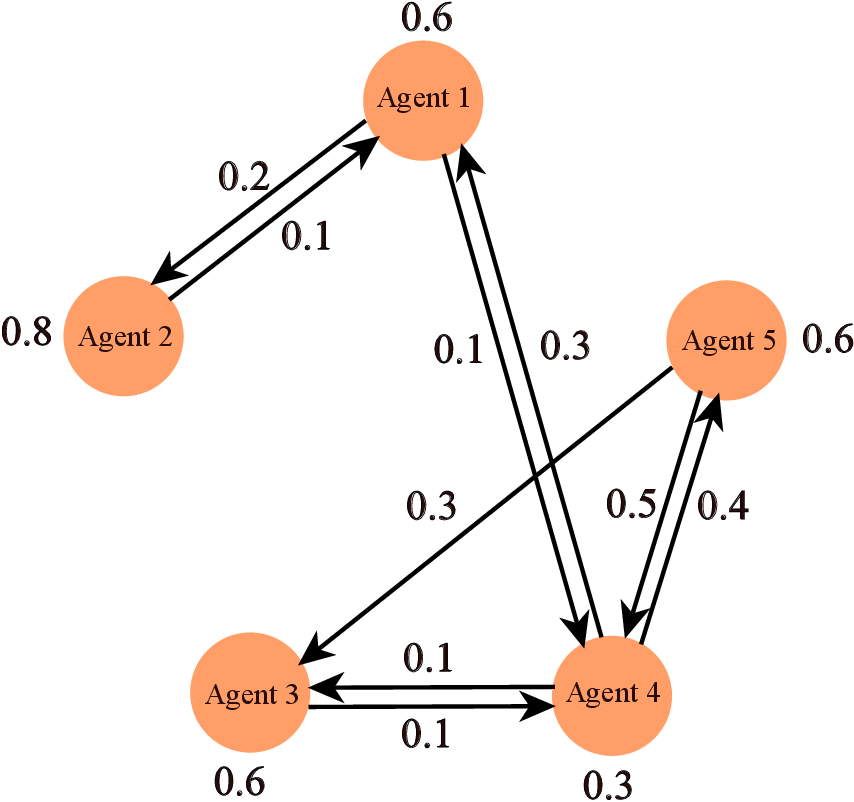}
    \caption{Column-stochastic mixing matrix $\boldsymbol{C}$.}
    \label{Fig:C}
\end{subfigure}
\caption{The directed communication graphs used in the evaluation.}
\label{Fig:RC}
\end{figure}

\subsection{Regularized Logistic Regression}\label{Sec:LR}
In this experiment, we employ nonconvex regularized logistic regression to solve
a binary classification problem using a real-world dataset from LIBSVM~\cite{chang2011libsvm},
specifically, the \emph{a9a} dataset. The feature vectors of the training samples are
denoted by $\boldsymbol{h} \in \mathbb{R}^{d}$, where $d = 123$, and the class
labels are $y \in \{0,1\}$. 
% In the implementation, the labels are converted from
% $y \in \{0,1\}$ to $y \in \{-1,1\}$ for computational convenience.

The loss function is defined as
\begin{equation}
\begin{aligned}\label{Equ:RLR}
&f_i(\boldsymbol{x}_i; \{\boldsymbol{h}, y\})   \\
 =&  -y \log( \frac{1}{1 + \exp(\boldsymbol{x}_i^\top \boldsymbol{h})}) + (1 - y) \log( \frac{\exp(\boldsymbol{x}_i^\top \boldsymbol{h})}{1 + \exp(\boldsymbol{x}_i^\top \boldsymbol{h})})  \\
 & +   \lambda_i \|\boldsymbol{x}_i\|^{p_i}, 
\end{aligned}
\end{equation}
where \(\{\boldsymbol{h}, y\}\) represents a training tuple, and \(\lambda_i\) denotes the regularization coefficient of agent \(i\). 

In the experiment, to reflect the heterogeneity in local data distributions and model preferences across the agents, we assign the following values for the five agents: $\
\lambda_1 = 5 \times 10^{-4},\ \lambda_2 = 1 \times 10^{-3},\ \lambda_3 = 2 \times 10^{-3},\ \lambda_4 = 1 \times 10^{-3},\ \lambda_5 = 1 \times 10^{-3},
$
$\ 
p_1 = 4,\ p_2 = 5,\ p_3 = 6,\ p_4 = 5,\ p_5 = 4.$  The regularization term $\|\boldsymbol{x}_i\|^{p_i}$ makes the loss function satisfy the
$(L_0, L_1)$-smoothness condition but not the conventional smoothness condition, as discussed in~\cite{vankov2024optimizing}.

\begin{figure}
\centerline{\includegraphics[width=\columnwidth]{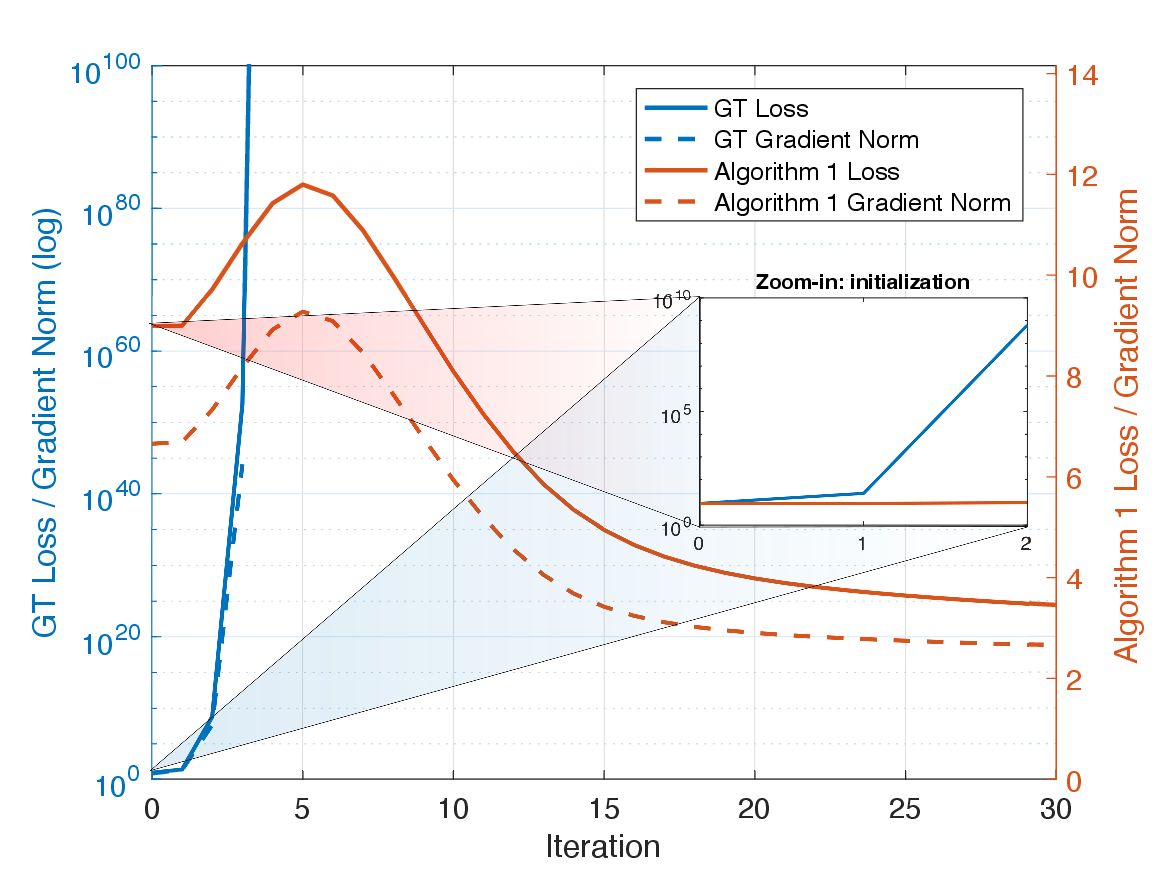}}
\caption{Comparison of loss and gradient norm between
Algorithm \ref{alg:1} and the gradient tracking algorithm in~\cite{xin2018linear} on the \emph{a9a} dataset. The standard gradient tracking (GT) method in~\cite{xin2018linear} 
(blue curves, left axis) exhibits severe instability during the initial iterations,
where both the loss and gradient norm rapidly explode. In contrast, Algorithm~\ref{alg:1} (red curves, right axis) ensures a smooth decrease in both loss value and gradient norm.  The zoom-in subplot highlights that both algorithms start from the
same initialization.}
\label{Fig:loss_grad_logistic}
\end{figure}

\begin{figure}
\centerline{\includegraphics[width=\columnwidth]{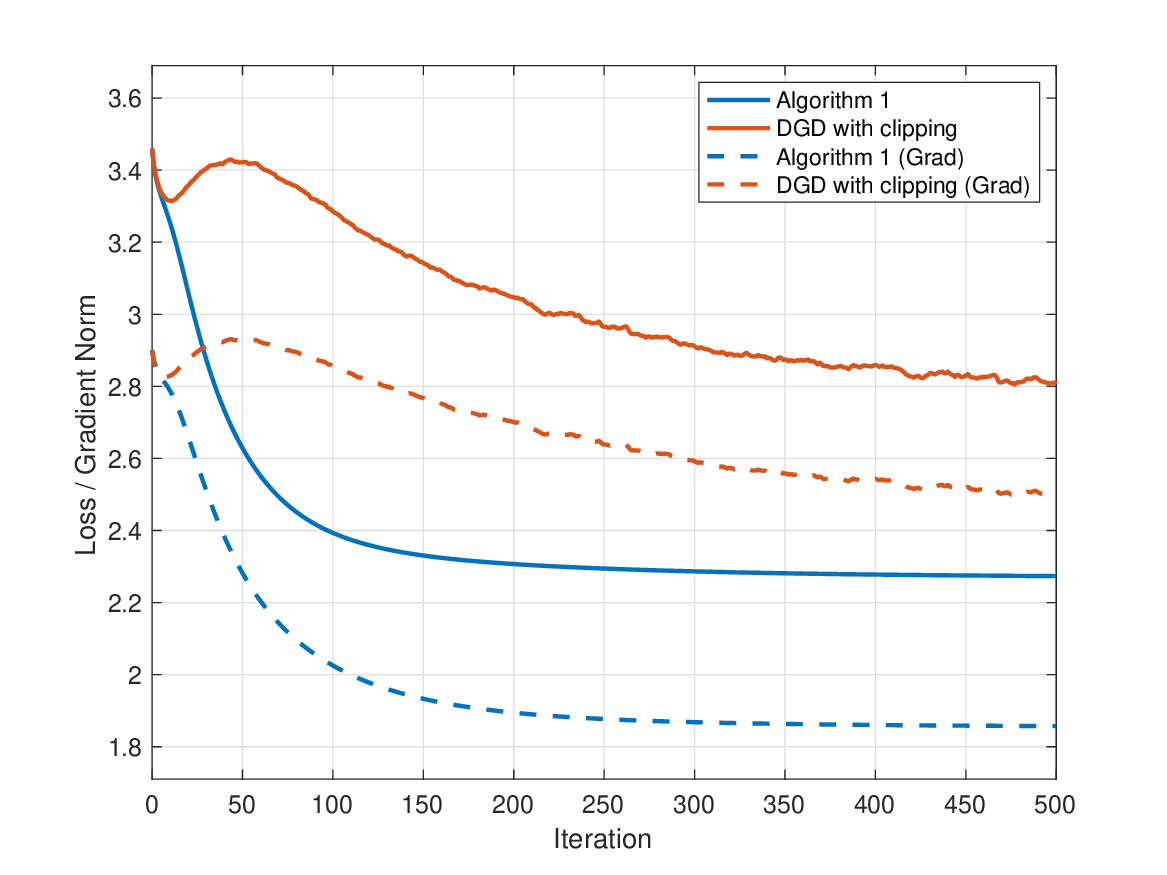}}
\caption{Comparison of loss and gradient norm evolution for Algorithm~\ref{alg:1} and DGD with gradient clipping~\cite{sunclipping} on the \emph{a9a} dataset.}
\label{Fig:DGD_compared}
\end{figure}
% \begin{figure}[t]
% \centering
% \begin{minipage}{0.48\textwidth}
%     \centering
%     \includegraphics[width=0.78\linewidth]{Loss.eps} 
%     \vspace{-3pt}
%     \caption{Loss Function Comparison for Regularized Logistic Regression (a9a Dataset)}
%     \label{Fig:loss_logistic}
% \end{minipage}\hfill
% \begin{minipage}{0.48\textwidth}
%     \centering
%     \includegraphics[width=0.78\linewidth]{Grad_norm.eps}  
%     \vspace{-3pt}
%     \caption{Gradient Norm Comparison during Logistic Regression Training (a9a Dataset)}
%     \label{Fig:grad_logistic}
% \end{minipage}
% \vspace{-6pt}
% \end{figure}

We compare the performance of the proposed Algorithm~\ref{alg:1} with the
standard gradient tracking method~\cite{xin2018linear} and the decentralized gradient descent (DGD) with clipping~\cite{sunclipping}. In all algorithms, the batch size is set to $32$, and the stepsize is fixed as $\alpha = 0.05$. For the clipping-based methods, the clipping threshold is
chosen as $c_0 = 5$.

Fig.~\ref{Fig:loss_grad_logistic} presents the evolution of the loss function and the
gradient norm for the standard gradient tracking algorithm~\cite{xin2018linear} and the proposed Algorithm~\ref{alg:1}. The standard gradient tracking method (blue curves)
exhibits severe instability during the initial iterations. This instability arises from large gradient magnitudes induced by $(L_0, L_1)$-smoothness.
In contrast, Algorithm~\ref{alg:1} (red curves) remains stable
throughout the training process. The clipping mechanism effectively controls the
magnitude of gradient updates during the early iterations, preventing the explosion
observed in the standard gradient tracking algorithm~\cite{xin2018linear}. 

Fig.~\ref{Fig:DGD_compared} illustrates the evolution of loss and gradient norms under the proposed Algorithm \ref{alg:1} and the algorithm in~\cite{sunclipping}, which is based on DGD with gradient clipping. It is evident that our proposed algorithm achieves fast and stable convergence, whereas DGD with gradient clipping exhibits pronounced oscillations and a significantly slower convergence rate. This highlights the advantages of our algorithm design and confirms the issue discussed earlier, namely that directly clipping local gradients can cause problems when different agents have heterogeneous objective functions.

\subsection{Convolutional Neural Network}

For this experiment, we consider the training of a convolutional neural network (CNN) for the classification of the CIFAR-10 dataset~\cite{krizhevsky2009learning}, which contains 50,000 training images across 10 different classes. We evenly spread the CIFAR-10 dataset among the five agents and set the batch size to 32. Our baseline CNN architecture is a deep network, ResNet-18, the training of which is a highly nonconvex and non-Lipschitz continuous problem. 

In the experiments, we train the CNN using the proposed Algorithm~\ref{alg:1} and compare its performance with several representative distributed optimization algorithms, including the standard gradient tracking (DGT)~\cite{xin2018linear}, CDSGD~\cite{jiang2017collaborative}, CDSGD with Polyak momentum (CDSGD-P)~\cite{jiang2017collaborative}, CDSGD with Nesterov momentum (CDSGD-N)~\cite{jiang2021consensus}, DAMSGrad~\cite{chen2023convergence}, and DAdaGrad~\cite{chen2023convergence}. For the proposed algorithm, the stepsize and clipping threshold were set to $\alpha = 0.05$ and $c_0 = 10$, respectively. For all baseline algorithms, the largest stepsizes that ensure convergence were adopted to provide a fair comparison.

\begin{figure}[t]
\centering
\begin{minipage}{0.48\textwidth}
    \centering
    \includegraphics[width=\linewidth]{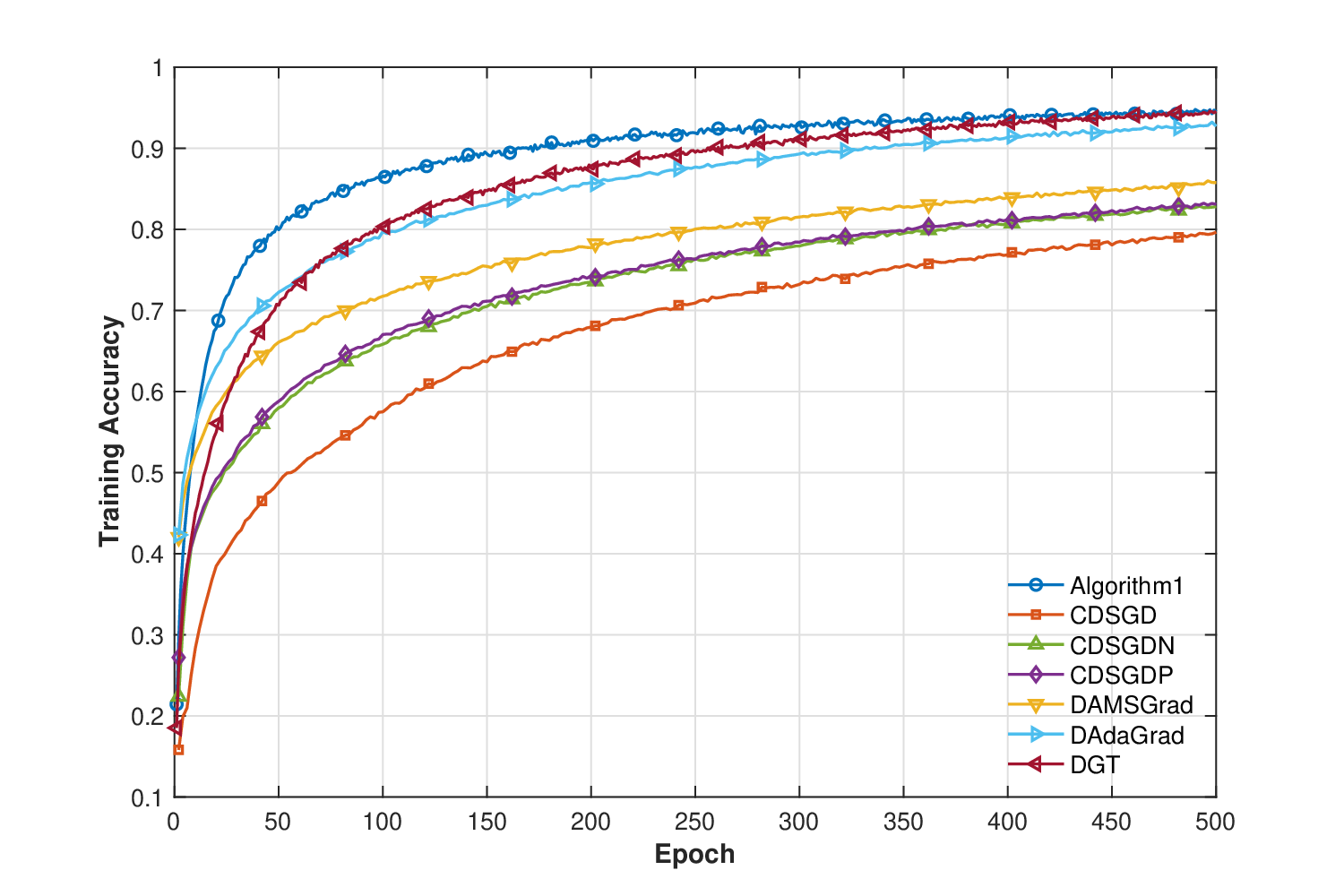} 
    \vspace{-3pt}
    \caption{Comparison of the proposed algorithm with state-of-the-art methods in terms of training accuracy on the CIFAR-10 dataset.}
    \label{Fig: CNN_train}
\end{minipage}\hfill
\begin{minipage}{0.48\textwidth}
    \centering
    \includegraphics[width=\linewidth]{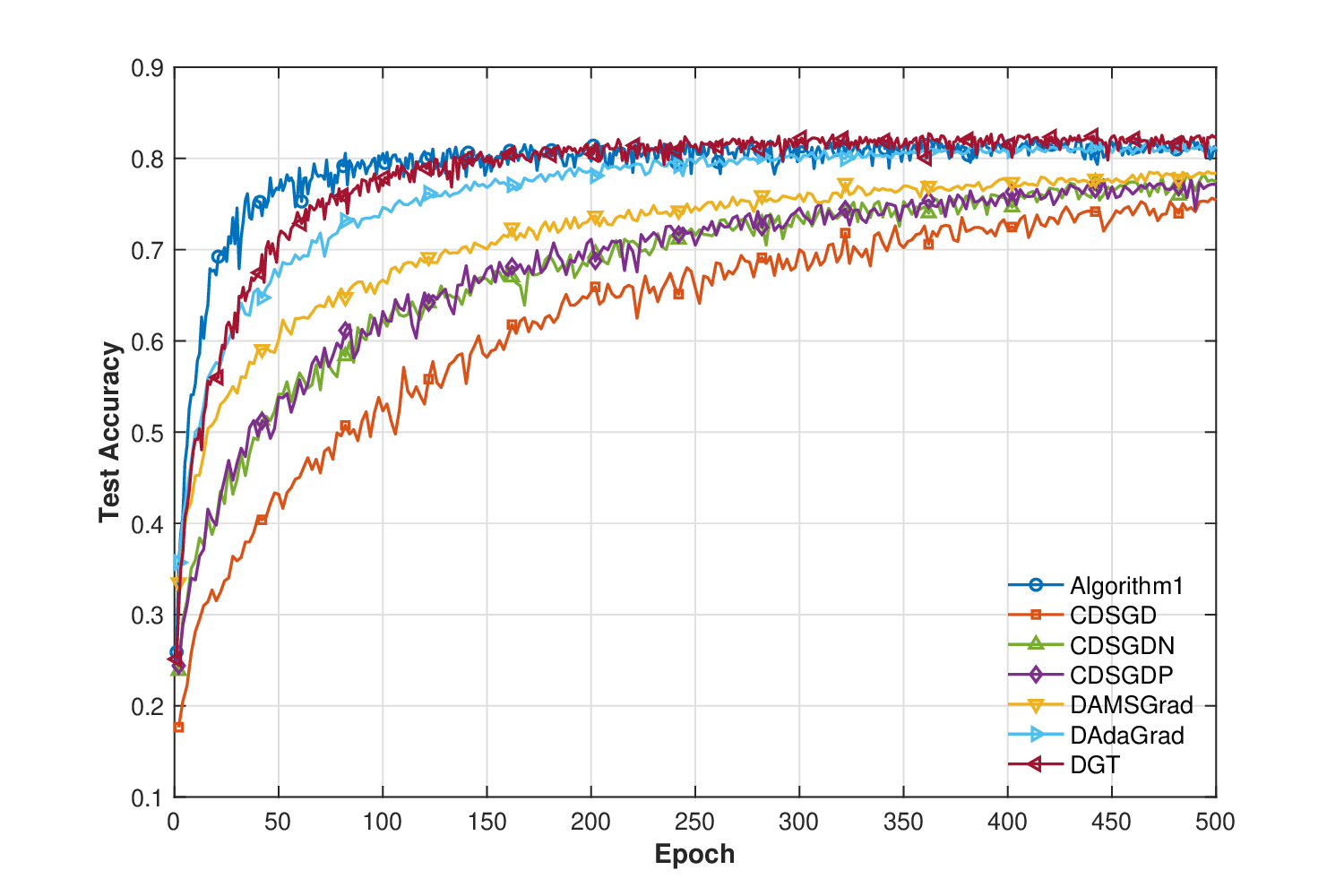}  
    \vspace{-3pt}
    \caption{Comparison of the proposed algorithm with state-of-the-art methods in terms of test accuracy on the CIFAR-10 dataset.}
    \label{fig:CNN_test}
\end{minipage}
\vspace{-6pt}
\end{figure}

% The training and test accuracies are averaged over 10 independent runs. 
The evolutions of the training accuracy and test accuracy are illustrated in Fig.~\ref{Fig: CNN_train} and Fig.~\ref{fig:CNN_test}, respectively. As shown in Fig.~\ref{Fig: CNN_train}, the proposed algorithm exhibits a faster convergence rate and achieves higher training accuracy than existing state-of-the-art distributed optimization methods. Moreover, Fig.~\ref{fig:CNN_test} demonstrates that the proposed algorithm consistently attains superior test accuracy to the counterpart algorithms, highlighting its strong generalization ability. These results confirm the effectiveness of the proposed approach for decentralized deep learning on nonconvex and generalized smoothness problems.

% \section{Conclusions}\label{Sec:conclusion}

% In this work, we develop a distributed optimization algorithm with provable convergence guarantees over directed communication graphs under the generalized smoothness condition. Moreover, we relax the commonly used bounded gradient dissimilarity assumption that appears in most existing decentralized optimization studies, thereby making our framework more applicable to realistic heterogeneous data environments.  By carefully designing algorithmic parameters and developing refined inequalities tailored
% to our update structure, we established the convergence guarantees for the proposed method and showed that the algorithm achieves an 
% $\mathcal{O}(\epsilon^{-2})$ convergence rate, which matches the best-known complexity bound of centralized algorithms under the same smoothness condition. Extensive experiments on benchmark datasets validate the effectiveness and efficiency of the proposed algorithm. Future work includes extending the analysis to stochastic gradient settings and time-varying network topologies.
\section{Conclusions}\label{Sec:conclusion}

In this work, we have proposed a new distributed optimization algorithm that can ensure accurate convergence under directed communication graphs and $(L_0, L_1)$-smooth objective functions that do not necessarily satisfy the conventional smoothness condition.  
 Unlike existing results for $(L_0, L_1)$-smoothness that rely on bounded gradient dissimilarity, our approach ensures accurate convergence even when the gradient dissimilarity is unbounded. A key innovation is to apply clipping to local estimates of the global gradient rather than to the local gradients directly. This, however, introduces significant nonlinearity and complexity in the convergence analysis, rendering conventional analysis techniques inapplicable. To address this, we established a new theoretical framework that provides rigorous convergence guarantees. 
 In fact, our analysis establishes that the algorithm achieves an $\mathcal{O}(\epsilon^{-2})$ convergence rate, matching existing results for centralized methods under the same smoothness condition. Numerical experiments on real-world datasets further confirm the effectiveness of the proposed approach.
\appendices

\section*{Appendix A: Proof of lemma \ref{Lemma:global_generalized}}\label{appendix:global_smoothness}

From Assumption~\ref{Assum:Lipschitz}, we have
\begin{equation}
\begin{aligned}
    \bigl\| \nabla^2 F(\boldsymbol{\theta}) \bigr\|
    &= \frac{1}{N} \left\| \sum_{i=1}^N \nabla^2 f_i(\boldsymbol{\theta}) \right\|  \\
    &\leqslant\frac{1}{N} \sum_{i=1}^N \left( L_0^i + L_1^i \bigl\| \nabla f_i(\boldsymbol{\theta}) \bigr\| \right).
\end{aligned}
\end{equation}

Next, using Assumption~\ref{Assum:gradient_dissimilarity}, we obtain
\begin{equation}
\begin{aligned}
    \bigl\| \nabla^2 F(\boldsymbol{\theta}) \bigr\|
    \leqslant&\frac{1}{N} \sum_{i=1}^N \left( 
        L_0^i + L_1^i \left( \ell \, \bigl\| \nabla F(\boldsymbol{\theta}) \bigr\| + b \right)
    \right) \\
    =& \frac{1}{N} \sum_{i=1}^N \left( L_0^i + b L_1^i \right)\\
    &+ \left( \frac{\ell}{N} \sum_{i=1}^N L_1^i \right)
      \bigl\| \nabla F(\boldsymbol{\theta}) \bigr\|.
\end{aligned}
\end{equation}

Therefore, the global objective \(F(\boldsymbol{\theta})\) is also \((L_0, L_1)\)-smooth, where
\[
L_0 = \frac{1}{N} \sum_{i=1}^N \left( L_0^i + b L_1^i \right),
\qquad
L_1 = \frac{\ell}{N} \sum_{i=1}^N L_1^i.
\]

% \section{Properties of Mixing Matrices}
% \begin{assumption}\label{Assum:graphs_RC}
% The graphs $\mathcal{G}_{\mathbf{R}}$ and $\mathcal{G}_{\mathbf{C}^\top}$ each contain at least one spanning tree. Moreover, there exists at least one node that is a root of spanning trees for both $\mathcal{G}_{\mathbf{R}}$ and $\mathcal{G}_{\mathbf{C}^\top}$, i.e., 
% $\mathcal{R}_{\mathbf{R}} \cap \mathcal{R}_{\mathbf{C}^\top} \neq \emptyset$, 
% where $\mathcal{R}_{\mathbf{R}}$ (resp., $\mathcal{R}_{\mathbf{C}^\top}$) is the set of roots of all possible spanning trees in the graph $\mathcal{G}_{\mathbf{R}}$ (resp., $\mathcal{G}_{\mathbf{C}^\top}$).
% \end{assumption}

% \begin{lemma}
% Under Assumption~\ref{Assum:matrices_RC}, the matrix $\mathbf{R}$ has a nonnegative left eigenvector $\boldsymbol{u}^\top$ 
% (w.r.t.\ eigenvalue $1$) with $\boldsymbol{u}^\top \mathbf{1} = N$. Similarly, the matrix $\mathbf{C}$ has a 
% strictly positive right eigenvector $\boldsymbol{v}$ (w.r.t.\ eigenvalue $1$) with $\mathbf{1}^\top \boldsymbol{v} = N$.
% \end{lemma}

\section*{Appendix B: Some useful lemmas}\label{Appendix_Properties of $(L_0,L_1)$ functions and Useful Lemmas}

% \begin{lemma}[Descent inequality in \cite{zhang2020improved}]\label{Lemma:pre3}
% Let $f$ be $(L_0,L_1)$-smooth, and let $c>0$ be a constant. For any 
% $\boldsymbol{x}_k$ and $\boldsymbol{x}_{k+1}$ satisfying 
% $\|\boldsymbol{x}_{k+1}-\boldsymbol{x}_k\| \leqslant c/L_1$, we have
% \begin{equation}
% \label{Eq:pre3}
% \begin{aligned}
% f(\boldsymbol{x}_{k+1})
% &\leqslant 
% f(\boldsymbol{x}_k)
%  + \big\langle \nabla f(\boldsymbol{x}_k),\, \boldsymbol{x}_{k+1}-\boldsymbol{x}_k \big\rangle 
% \\
% &\quad{}\; 
% + \frac{A L_0 + B L_1 \|\nabla f(\boldsymbol{x}_k)\|}{2}\,
%   \|\boldsymbol{x}_{k+1}-\boldsymbol{x}_k\|^2,
% \end{aligned}
% \end{equation}
% where 
% \[
% A = 1 + e^{c} - \frac{e^{c}-1}{c},
% \qquad
% B = \frac{e^{c}-1}{c}.
% \]
% \end{lemma}

% \begin{lemma}[\cite{zhang2020improved}]\label{Lemma:pre4}
% Let $f$ be $(L_0,L_1)$-smooth, and let $c>0$ be a constant. For any 
% $\boldsymbol{x}_k$ and $ \boldsymbol{x}_{k+1}$ satisfying 
% $\|\boldsymbol{x}_{k+1}-\boldsymbol{x}_k\|\leqslant  c/L_1$, it holds that
% \begin{equation}\label{Eq:pre4}
% \begin{aligned}
% &\|\nabla f(\boldsymbol{x}_{k+1}) - \nabla f(\boldsymbol{x}_k)\|\\
% \leqslant & 
% \Bigl( A L_0 + B L_1 \|\nabla f(\boldsymbol{x}_k)\| \Bigr)\,
%       \|\boldsymbol{x}_{k+1}-\boldsymbol{x}_k\|.
% \end{aligned}
% \end{equation}

% where
% \[
% A = 1 + e^{c} - \frac{e^{c}-1}{c},
% \qquad
% B = \frac{e^{c}-1}{c}.
% \]
% \end{lemma}

\begin{lemma}[\cite{li2023convex}, Lemma 3.5]\label{Lemma:sub-optimality gap}
If $f$ is $(L_0,L_1)$-smooth, then for any $\boldsymbol{x} \in \mathbb{R}^d$, we have
\begin{equation}\label{Eq:sub-opt-gap}
\|\nabla f(\boldsymbol{x})\|^{2}
\;\leqslant \;
2\,\bigl(L_0 + 2L_1\|\nabla f(\boldsymbol{x})\|\bigr)\, 
\bigl( f(\boldsymbol{x}) - \underline{f} \bigr).
\end{equation}
\end{lemma}
\begin{lemma}[\cite{li2023convex}, 
Corollary 3.6]\label{Lemma:F->gradient}
Suppose $f$ is $(L_0,L_1)$-smooth.  
If for some $\boldsymbol{x} \in \mathcal{X}$, we have 
$f(\boldsymbol{x}) - \underline{f} \leqslant \Delta_f$ with $\Delta_f \geqslant 0$, then we have 
\[
G^2 = 2\bigl(L_0 + 2L_1\|\nabla f(\boldsymbol{x})\|\bigr)\,\Delta_f \] and 
\[
\|\nabla f(\boldsymbol{x})\| \leqslant G < \infty,
\]
where 
\[
G = 
\sup\left\{
u \geqslant 0 
\;\bigg|\;
u^{2} \leqslant 2\bigl(L_0 + 2L_1\|\nabla f(\boldsymbol{x})\|\bigr)\, \Delta_f
\right\}.
\]
\end{lemma}

\section*{Appendix C: Proof of Lemma~\ref{Lemma:pre1}}\label{Appendix:proof_Lemma:pre1}
% \begin{lemma}\label{Lemma:pre1}
% For any agent $i\in [N]$ in Algorithm~\ref{alg:1}, define
% \[
% \alpha_i^k
% = \alpha \min \left\{ 1, \tfrac{c_0}{\| \boldsymbol{y}_i^k \|} \right\},
% \]
% \[
% \bar{\alpha}_i^k
% = \alpha \min \left\{ 1, \tfrac{c_0}{v_i\| \nabla F(\bar{\boldsymbol{x}}^k) \|} \right\}.
% \]
% Then, the following inequality holds under Assumption~\ref{Assum:lower_bound},
% Assumption~\ref{Assum:Lipschitz}, and Assumption~\ref{Assum:matrices_RC}:
% \begin{equation}\label{Eq:pre1}
% \bigl| \alpha _{i}^{k}-\bar{\alpha}_i^k \bigr| \| \boldsymbol{y}_{i}^{k}\| \leqslant \; \bar{\alpha}_i^k \| \boldsymbol{y}_{i}^{k}-v_i\nabla F(\bar{\boldsymbol{x}}^k)\| .
% \end{equation}
% Furthermore, if we define $\bar{\alpha}_k
% = \alpha \min \left\{ 1, \tfrac{c_0}{\| \boldsymbol{v}\nabla F(\bar{\boldsymbol{x}}^k) \|} \right\}$, then we have $\bar{\alpha}_k \leqslant \bar{\alpha}_i^k \leqslant \frac{\|\boldsymbol{v}\|}{v_i}\bar{\alpha}_k$. Thus, we have
% \begin{equation}\label{Eq:pre1.1}
% \bigl| \alpha _{i}^{k}-\bar{\alpha}_{i}^{k} \bigr| \| \boldsymbol{y}_{i}^{k}\| \leqslant \;\frac{\| \boldsymbol{v}\|}{v_i}\bar{\alpha}_k\| \boldsymbol{y}_{i}^{k}-v_i\nabla F(\bar{\boldsymbol{x}}^k)\| .
% \end{equation}
% \end{lemma}

  According to the definitions of \(\alpha_i^k\) and \(\bar{\alpha}_i^k \), we establish the 
 relationship in Equation~(\ref{Eq:pre1}) on a case-by-case basis as follows:

\noindent\textbf{Case 1:}  \(v_i\| \nabla F(\bar{\boldsymbol{x}}^k) \|\leqslant c_0\), 
      \(\|\boldsymbol{y}_{i}^{k}\| \leqslant c_0\): \\[0.5ex]
In this case, we have \(\bar{\alpha}_i^k=\alpha_i^k=\alpha\), which implies
\[
   \bigl| \alpha_i^k-\bar{\alpha}_i^k \bigr|\,\|\boldsymbol{y}_{i}^{k}\| = 0.
\]

\medskip
\noindent\textbf{Case 2:}  \(v_i\| \nabla F(\bar{\boldsymbol{x}}^k) \|\leqslant c_0\), 
      \(\|\boldsymbol{y}_{i}^{k}\| > c_0\): \\[0.5ex]
In this case, we have \( \alpha_i^k<\alpha=\bar{\alpha}_i^k \), which implies
\begin{equation}\label{Eq:temp1}
\begin{aligned}
    \bigl| \alpha_i^k-\bar{\alpha}_i^k \bigr|\,\| \boldsymbol{y}_{i}^{k}\|
   &= \alpha c_0\left| \tfrac{1}{\|\boldsymbol{y}_{i}^{k}\|}-\tfrac{1}{c_0} \right| \| \boldsymbol{y}_{i}^{k}\| \\
   &= \alpha c_0\left( \tfrac{1}{c_0}-\tfrac{1}{\|\boldsymbol{y}_{i}^{k}\|} \right) \| \boldsymbol{y}_{i}^{k}\| \\
   &= \alpha \left( 1-\tfrac{c_0}{\|\boldsymbol{y}_{i}^{k}\|} \right) \| \boldsymbol{y}_{i}^{k}\| .
\end{aligned}
\end{equation}
In the second equality above, we have used the condition \(\left\| \boldsymbol{y}_{i}^{k} \right\| >c_0\) in this case. 

By substituting  $\alpha$ with $\bar{\alpha}_i^k$ and using the relation 
\(v_i\| \nabla F(\bar{\boldsymbol{x}}^k) \| \leqslant c_0\), we obtain
\begin{align*}
   \bigl| \alpha_i^k-\bar{\alpha}_i^k \bigr|\,\|\boldsymbol{y}_i^k\|
   \leqslant& \bar{\alpha}_i^k\bigl( \| \boldsymbol{y}_{i}^{k}\| - \|v_i\nabla F(\bar{\boldsymbol{x}}^k)\|  \bigr)\\
   \leqslant& \bar{\alpha}_i^k \| \boldsymbol{y}_{i}^{k} - v_i\nabla F(\bar{\boldsymbol{x}}^k) \|, 
\end{align*}
   
where the last inequality follows from $|\|a\|-\|b\||\leqslant\|a-b\|$.

\medskip
\noindent\textbf{Case 3:} \(v_i\| \nabla F(\bar{\boldsymbol{x}}^k) \| > c_0\), \(\left\| \boldsymbol{y}_{i}^{k} \right\| \leqslant 
 c_0\): \\
In this case, we have  \( \alpha_i^k = \alpha>\bar{\alpha}_i^k \), which implies
\begin{equation}
    \begin{aligned}
        \left| \alpha_i^k-\bar{\alpha}_i^k \right|\left\| \boldsymbol{y}_i^k \right\| &=\alpha c_0\left| \frac{1}{c_0}-\frac{1}{v_i\left\| \nabla F\left( \boldsymbol{\bar{x}}^k \right) \right\|} \right|\left\| \boldsymbol{y}_i^k \right\| 
        \\
        &=\alpha \left( 1-\frac{c_0}{v_i\left\| \nabla F\left( \boldsymbol{\bar{x}}^k \right) \right\|} \right) \left\| \boldsymbol{y}_i^k \right\|
        \\
        &=\frac{\alpha \left\| \boldsymbol{y}_{i}^{k} \right\|}{v_i\left\| \nabla F\left( \boldsymbol{\bar{x}}^k \right) \right\|}\left( \left\| v_i\nabla F\left( \boldsymbol{\bar{x}}^k \right) \right\| -c_0 \right),
        \\
    \end{aligned}
\end{equation} 
where in the second equality we have used the condition \(v_i\| \nabla F(\bar{\boldsymbol{x}}^k) \| > c_0\) in this case. Using the definition of $\bar{\alpha}_i^k$ and the triangle inequality $|\|a\|-\|b\||\leqslant\|a-b\|$, we have the following inequality:

\begin{equation}
    \begin{aligned}
        \left| \alpha_i^k-\bar{\alpha}_i^k \right|\left\| \boldsymbol{y}_i^k \right\| 
        &\leqslant \frac{\alpha c_0}{v_i\left\| \nabla F\left( \boldsymbol{\bar{x}}^k \right) \right\|}\left( v_i\left\| \nabla F\left( \boldsymbol{\bar{x}}^k \right) \right\| -\left\| \boldsymbol{y}_i^k \right\| \right)
        \\
        &\leqslant \bar{\alpha}_i^k\left\| \boldsymbol{y}_i^k - v_i\nabla F\left( \boldsymbol{\bar{x}}^k \right) \right\| .
        \\
    \end{aligned}
\end{equation} 

\medskip
\noindent\textbf{Case 4:}  \(v_i\| \nabla F(\bar{\boldsymbol{x}}^k) \|>c_0\), 
      \(\|\boldsymbol{y}_{i}^{k}\| > c_0\): \\[0.5ex]
In this case, we have \( \alpha_i^k < \alpha\), \(\bar{\alpha}_i^k< \alpha\), which leads to
\begin{align*}
     \bigl| \alpha_i^k-\bar{\alpha}_i^k \bigr|\,\|\boldsymbol{y}_i^k\|
   =& \alpha c_0\left| \tfrac{1}{\|\boldsymbol{y}_{i}^{k}\|}-\tfrac{1}{v_i\| \nabla F(\bar{\boldsymbol{x}}^k)\|} \right|
     \|\boldsymbol{y}_i^k\|\\
   =& \alpha c_0\frac{\bigl|v_i\| \nabla F(\bar{\boldsymbol{x}}^k)\|-\|\boldsymbol{y}_{i}^{k}\|\bigr|}
                        {v_i\| \nabla F(\bar{\boldsymbol{x}}^k)\|}.
\end{align*}
  
Using $\bar{\alpha}_i^k = \tfrac{\alpha c_0}{v_i\|\nabla F(\bar{\boldsymbol{x}}^k)\|}$ and the triangle inequality $|\|a\|-\|b\||\leqslant\|a-b\|$, we obtain
\[
   \bigl| \alpha_i^k-\bar{\alpha}_i^k \bigr|\,\|\boldsymbol{y}_i^k\|
   \leqslant \bar{\alpha}_i^k\|\boldsymbol{y}_i^k - v_i\nabla F\left( \boldsymbol{\bar{x}}^k \right)\|.
\]

Therefore, the relationship in equation~(\ref{Eq:pre1}) is true in all cases, which completes the proof.

% \subsection{Bounding Consensus Errors}\label{PF:consensus}

% \begin{lemma}\label{Lemma:unibound_x}
%      Suppose that Assumptions~\ref{Assum:lower_bound}, \ref{Assum:Lipschitz} and \ref{Assum:matrices_RC} hold. 
% Under Algorithm~\ref{alg:1}, the optimization error $\|\boldsymbol{e}_{x,k}\|_R^2$ is uniformly bounded for all iterations. 
% In particular, we have
% \begin{equation}\label{eq:e_x_bound}
%     \|\boldsymbol{e}_{x,k}\|_R^2
%     \;\leq\;
%     \mathcal{C}_x\alpha^2c_0^2,\qquad \forall k \geqslant  0,
% \end{equation}
% where $\mathcal{C}_{x}=\frac{2\sigma_R^2 \big(1+\sigma_R^2\big)}{(1-\sigma_R^2)^2}
%     \, N ^2  \delta_{R,2}^2\| \boldsymbol{I}-\tfrac{\boldsymbol{1}\boldsymbol{u}^{\top}}{N}\|^2 _R$.
% \end{lemma}
\section*{ Appendix D: Proof of Lemma~\ref{Lemma:unibound_x}}\label{Appendix:proof_Lemma:unibound_x}
 According to the derivation of the optimization error $\boldsymbol{e}_{x,k}$ in (\ref{Eq:e_x^k}), by taking norm on the both sides of (\ref{Eq:e_x^k}) and utilizing the inequality $(a+b)^2\leqslant(1+\eta)a^2+(1+\frac{1}{\eta})b^2$, we get 
 \begin{equation}\label{Eq:norm_e_x}
 \begin{aligned}
    \|\boldsymbol{e}_{x,k+1}\|_R^2 
    & \leqslant (1+\eta) \|\big( \boldsymbol{R} - \tfrac{\boldsymbol{1}\boldsymbol{u}^\top}{N} \big)\boldsymbol{e}_{x,k}\|_R^2 \\
       & \quad + (1+\tfrac{1}{\eta}) \|\big( \boldsymbol{I} - \tfrac{\boldsymbol{1}\boldsymbol{u}^\top}{N} \big)\boldsymbol{\alpha}_k \boldsymbol{y}^k\|_R^2, \\[0.5em]
    &\leqslant (1+\eta)\sigma_R^2\delta_{R,2}^2 \|\boldsymbol{e}_{x,k}\|^2 \\
        & \quad + \Big(1+\tfrac{1}{\eta}\Big) \| \boldsymbol{I} - \tfrac{\boldsymbol{1}\boldsymbol{u}^\top}{N} \|^2_R \delta_{R,2}^2 \|\boldsymbol{\alpha}_k \boldsymbol{y}^k\|^2. \\[0.5em]
\end{aligned}
\end{equation}
By the definition $\boldsymbol{\alpha}^k \boldsymbol{y}^k = [({\alpha}_1^k \boldsymbol{y}_1^k)^\top; \cdots; ({\alpha}_N^k \boldsymbol{y}_N^k)^\top] \in \mathbb{R}^{N \times d}$, we get
 \begin{equation}\label{Eq:Nalphac0}
 \begin{aligned}
    \|\boldsymbol{\alpha}_k \boldsymbol{y}^k\|^2
    &= \sum_{i=1}^N (\alpha_i^k)^2 \|\boldsymbol{y}_i^k\|^2 \\[0.5em]
    &=\sum_{i=1}^N \alpha^2 \min\Big\{1, \tfrac{c_0^2}{\|\boldsymbol{y}_i^k\|^2}\Big\} \|\boldsymbol{y}_i^k\|^2 \\[0.5em]
    &\leqslant  N\alpha^2 c_0^2.
\end{aligned}
\end{equation}
Combing (\ref{Eq:norm_e_x}) and (\ref{Eq:Nalphac0}), we have 
\begin{equation}\label{Eq:recursion_x}
 \begin{aligned}
    &\|\boldsymbol{e}_{x,k+1}\|_R^2 \\
     = &(1+\eta)\sigma_R^2\delta_{R,2}^2 \|\boldsymbol{e}_{x,k}\|^2 
        + (1+\tfrac{1}{\eta}) \| \boldsymbol{I} - \tfrac{\boldsymbol{1}\boldsymbol{u}^\top}{N} \|^2_R N \alpha^2 c_0^2 \delta_{R,2}^2 \\
    \leqslant &\tfrac{1+\sigma_R^2}{2} \|\boldsymbol{e}_{x,k}\|_R^2 + \tfrac{\sigma_R^2(1+\sigma_R^2)}{1-\sigma_R^2}\| \boldsymbol{I} - \tfrac{\boldsymbol{1}\boldsymbol{u}^\top}{N} \|^2_RN \alpha^2 c_0^2 \delta_{R,2}^2,
\end{aligned}
\end{equation}
where the last inequality has used $\eta=\frac{1-\sigma_R^2}{2\sigma_R^2}$.

The inequality in \eqref{Eq:recursion_x} is a linear recursion of the form
\[
    z_{k+1} \leqslant a' z_k + b',
\]
where
\[
    z_k = \|\boldsymbol{e}_{x,k}\|_R^2,
    \qquad
    a' = \tfrac{1+\sigma_R^2}{2},
    \qquad
    b' = \tfrac{\sigma_R^2(1+\sigma_R^2)}{1-\sigma_R^2}N \alpha^2 c_0^2 \delta_{R,2}^2.
\]

By iterating (\ref{Eq:recursion_x}), we obtain
\begin{align}
    \|\boldsymbol{e}_{x,k+1}\|_R^2
    &\leqslant a'^{k+1} \|\boldsymbol{e}_{x,0}\|_R^2
    + \frac{b'}{1-a'}\big(1-a'^{k+1}\big).
\end{align}

In particular, if $a' < 1$, the consensus error is uniformly bounded, i.e.,
\begin{equation}\label{eq:constant_bound}
    \|\boldsymbol{e}_{x,k}\|_R^2
    \;\leqslant \;\frac{2N\sigma_R^2 (1+\sigma_R^2)\delta_{R,2}^2 \|\boldsymbol{I}-\tfrac{\boldsymbol{1}\boldsymbol{u}^{\top}}{N}\|_R^2}{(1-\sigma_R^2)^2} \alpha ^2c_{0}^{2}
,
    \quad  \forall k \geqslant  0.
\end{equation}

\section*{Appendix E: Proof of Lemma~\ref{Lemma:unibound_y}}\label{Appendix:proof_Lemma:unibound_y}

    According to (\ref{Eq:e_y^k1}), by taking the norm on both sides of the inequality, we have
\begin{equation}\label{Eq:norm_e_y}
\begin{aligned}
& \left\| \boldsymbol{e}_{y,k+1} \right\| _C^{2} \\
\leqslant & (1+\eta )\| \bigl( \boldsymbol{C}-\tfrac{\boldsymbol{v}\boldsymbol{1}^{\top}}{N} \bigr) \boldsymbol{e}_{y,k}\| _C^{2}\\
& +(1+\tfrac{1}{\eta})\| \bigl( \boldsymbol{I}-\tfrac{\boldsymbol{v}\boldsymbol{1}^{\top}}{N} \bigr) \bigl( \nabla f(\boldsymbol{x}^{k+1})-\nabla f(\boldsymbol{x}^k) \bigr) \| _C^{2}
\\
\leqslant &(1+\eta )\sigma _{C}^{2}\| \boldsymbol{e}_{y,k}\| _C^{2}\\
& +(1+\tfrac{1}{\eta})\left\| \boldsymbol{I}-\tfrac{\boldsymbol{v}\boldsymbol{1}^{\top}}{N} \right\| _C^{2}\| \nabla f(\boldsymbol{x}^{k+1})-\nabla f(\boldsymbol{x}^k)\| _C^{2}\\
    \leqslant & \tfrac{1+\sigma _{C}^{2}}{2}\| \boldsymbol{e}_{y,k}\| _C^{2}
     +\tfrac{1+\sigma _{C}^{2}}{1-\sigma _{C}^{2}}\delta _{C,2}^{2}\left\| \boldsymbol{I}-\tfrac{\boldsymbol{v}\boldsymbol{1}^{\top}}{N} \right\| _C^{2}\\
     &\| \nabla f(\boldsymbol{x}^{k+1})-\nabla f(\boldsymbol{x}^k)\| ^2,
    \end{aligned}
\end{equation}
where the last inequality follows from the choice of $\eta=\frac{1-\sigma_C^2}{2\sigma_C^2}$.

For the second term $\|\nabla f(\boldsymbol{x}^{k+1})-\nabla f(\boldsymbol{x}^{k})\|^2$, we can divide it into two parts as follows:

\begin{equation}
\begin{aligned}\label{Eq:e_y_1}
    &\|\nabla f(\boldsymbol{x}^{k+1}) - \nabla f(\boldsymbol{x}^{k})\|^2\\
    \leqslant & 2\|\nabla f(\boldsymbol{x}^{k+1}) - \nabla f(\mathbf{1}\bar{\boldsymbol{x}}^{k})\|^2
    + 2\|\nabla f(\mathbf{1}\bar{\boldsymbol{x}}^{k}) - \nabla f(\boldsymbol{x}^{k})\|^2. \\[0.5em]
\end{aligned}
\end{equation}

In order to use the property of $(L_0, L_1)$-smooth function in Lemma~\ref{Lemma:pre4}, we first analyze the term 
$
    \|\boldsymbol{x}^{k+1} -\mathbf{1}\bar{\boldsymbol{x}}^{k}\|^2.
$
Using (\ref{Eq:update_x}), we have
\begin{equation}
\begin{aligned}\label{Eq:e_y_2.1}
    &\|\boldsymbol{x}^{k+1} - \mathbf{1}\bar{\boldsymbol{x}}^{k}\|^2\\
    =& \|\boldsymbol{R}\boldsymbol{x}^k + \boldsymbol{\alpha}_k \boldsymbol{y}^k - \mathbf{1}\bar{\boldsymbol{x}}^k\|^2 \\[0.5em]
    \leqslant&  2 \|\boldsymbol{R}\boldsymbol{x}^k - \mathbf{1}\bar{\boldsymbol{x}}^k\|_R^2
        + 2 \|\boldsymbol{\alpha}_k \boldsymbol{y}^k\|^2  \\[0.5em]
        \leqslant&  2 \left\|  \boldsymbol{R}- \frac{\mathbf{1}u^\top}{N} \right\|_R^2 \|\boldsymbol{x}^k - \mathbf{1}\bar{\boldsymbol{x}}^k\|_R^2
        + 2 \|\boldsymbol{\alpha}_k \boldsymbol{y}^k\|^2. \\[0.5em]
   \end{aligned}
\end{equation}
By Lemma~\ref{Lemma:sigma_R,C},  we obtain an upper bound of $\|\boldsymbol{x}^{k+1} - \mathbf{1}\bar{\boldsymbol{x}}^{k}\|^2$ as follows:
\begin{equation}
\begin{aligned}\label{Eq:x_k+1-x_k_bar}
    \|\boldsymbol{x}^{k+1} - \mathbf{1}\bar{\boldsymbol{x}}^{k}\|^2
    &\leqslant   2 \sigma_{R}^2 \|\boldsymbol{e}_{x,k}\|_R^2
        + 2 N \alpha^2 c_0^2.  \\[0.5em]
\end{aligned}
\end{equation}
% Combining (\ref{eq:constant_bound}) and (\ref{Eq:e_y_2}), we obtain the upper bound of $\|\boldsymbol{x}^{k+1} - \mathbf{1}\bar{\boldsymbol{x}}^{k}\|^2$ as follows:
% \begin{equation}\label{Eq:x_k+1-x_k_bar}
% \begin{aligned}
%   &\|\boldsymbol{x}^{k+1} - \mathbf{1}\bar{\boldsymbol{x}}^{k}\|^2 \\
%   \leqslant&
%   2N\alpha ^2c_{0}^{2}\left( \tfrac{2\sigma _R^{4}\delta _{R,2}^{2}\left( 1+\sigma _R^{2} \right)\| \boldsymbol{I}-\tfrac{\boldsymbol{1}\boldsymbol{u}^{\top}}{N}\|_R
% }{\left( 1-\sigma _R^{2} \right) ^2}+1 \right) 
% \end{aligned}
% \end{equation}

Then, according to Lemma~\ref{Lemma:pre4}, we can get

\begin{equation}\label{Eq:grad_diff_2}
\begin{aligned}
& \quad \|\nabla f(\boldsymbol{x}^{k+1}) - \nabla f(\boldsymbol{x}^{k})\|^2\\[0.5em]
 &\leqslant 2\left(A L_0 + B L_1\|\nabla f_i(\bar{\boldsymbol{x}}^{k})\|\right)^2 
       \|\boldsymbol{x}^{k+1} - \mathbf{1}\bar{\boldsymbol{x}}^{k}\|^2  \\
        &\quad + 2\left(A L_0 + B L_1 \|\nabla f_i(\bar{\boldsymbol{x}}^{k})\|\right)^2 
       \|\boldsymbol{x}^{k} -\mathbf{1}\bar{\boldsymbol{x}}^{k} \|^2.\\
       \end{aligned}
\end{equation}
By Assumption~\ref{Assum:gradient_dissimilarity}, substituting $\|\nabla f_i(\bar{\boldsymbol{x}}^{k})\|$ with $\|\nabla f(\bar{\boldsymbol{x}}^{k})\|$, we have:
\begin{equation}\label{Eq:grad_diff_1}
\begin{aligned}
& \quad \|\nabla f(\boldsymbol{x}^{k+1}) - \nabla f(\boldsymbol{x}^{k})\|^2\\[0.5em]
    &\leqslant 2\left(A L_0 + B L_1b + B L_1 \ell \|\nabla f(\bar{\boldsymbol{x}}^{k})\|\right)^2 
       \|\boldsymbol{x}^{k+1} - \mathbf{1}\bar{\boldsymbol{x}}^{k}\|^2  \\
    &\quad + 2\left(A L_0 + B L_1b + B L_1 \ell \|\nabla f(\bar{\boldsymbol{x}}^{k})\|\right)^2 
       \|\boldsymbol{x}^{k} -\mathbf{1}\bar{\boldsymbol{x}}^{k} \|^2 \\[0.5em]
    &\leqslant 2\left(A L_0 + B L_1b + B L_1 \ell G\right)^2 \\
       & \quad \Big(\|\boldsymbol{x}^{k+1} - \mathbf{1}\bar{\boldsymbol{x}}^{k}\|^2 
            + \|\boldsymbol{x}^{k} - \mathbf{1}\bar{\boldsymbol{x}}^{k}\|^2\Big).
\end{aligned}
\end{equation}

By Lemma~\ref{Lemma:unibound_x} and (\ref{Eq:x_k+1-x_k_bar}), we have
\begin{equation}\label{Eq:grad_diff}
\begin{aligned}   
& \quad \|\nabla f(\boldsymbol{x}^{k+1}) - \nabla f(\boldsymbol{x}^{k})\|^2\\[0.5em]
            &\leqslant 2\left(AL_0+BL_1b +BL_1\ell G\right)^2\bigl(2N+(1+2\sigma_R^2)\mathcal{C}_{x}\bigr)\alpha ^2c_0^2 \\
            &=\mathcal{C} _1\alpha ^2c_0^2,
    \end{aligned}
\end{equation}
where  $\mathcal{C}_{1} = 2(AL_0 + BL_1 b + BL_1 \ell G)^{2} 
\bigl(2N+(1+2\sigma_R^2)\mathcal{C}_{x}\bigr)$.
 Combining (\ref{Eq:norm_e_y}) and  (\ref{Eq:grad_diff}), we obtain the  recursive relation for the gradient tracking error        $\left\| \boldsymbol{e}_{y,k} \right\| _C^{2}$ as follows:
\begin{equation}
    \begin{aligned}
        &\left\| \boldsymbol{e}_{y,k+1} \right\| _C^{2} \\ 
        \leqslant& \frac{1+\sigma _{C}^{2}}{2}\| \boldsymbol{e}_{y,k}\| _C^{2}+\frac{1+\sigma _{C}^{2}}{1-\sigma _{C}^{2}}\delta _{C,2}^{2}\left\| \boldsymbol{I}-\tfrac{\boldsymbol{v}\boldsymbol{1}^{\top}}{N} \right\| _C^{2}\mathcal{C} _1\alpha ^2 c_0^2.
        \end{aligned}
\end{equation}

Then, we can obtain a uniform bound on the gradient tracking error $\left\| \boldsymbol{e}_{y,k} \right\| _C^{2}$:
\begin{equation}
    \| \boldsymbol{e}_{y,k} \|_C^{2} 
    \;\leqslant\; 
    \frac{2(1+\sigma_{C}^{2})}{(1-\sigma_{C}^{2})^{2}} \,
        \delta_{C,2}^{2} \,
        \left\| \boldsymbol{I} - \tfrac{\boldsymbol{v}\boldsymbol{1}^{\top}}{N} \right\|_C^{2} \,
        \mathcal{C}_{1}\,\alpha^{2}c_0^2.
\end{equation}

\section*{Appendix F: Proof of Lemma~\ref{Th:consensus}}\label{Appendix:proof_Th:consensus}

\subsection{Preliminary results}
To prove Lemma~\ref{Th:consensus}, we first present a useful preliminary result.

\begin{lemma}\label{Lemma:pre2}
Under Assumption~\ref{Assum:Lipschitz} and
Assumption~\ref{Assum:matrices_RC}, and using Lemma~\ref{Lemma:pre1}, the
iterations of Algorithm~\ref{alg:1} can be verified to satisfy
\begin{equation}\label{eq:alpha_y_bound}
\begin{aligned}
&\left\| \boldsymbol{\alpha}_k \boldsymbol{y}^k \right\|^2\\
\leqslant\;& 6\kappa _{v}^{2}\bar{\alpha}_{k}^{2}\bigl\| \boldsymbol{y}^k-\boldsymbol{v}\nabla F(\bar{\boldsymbol{x}}^k) \bigr\| ^2 + 3N\bar{\alpha}_{k}^{2}\| \boldsymbol{v}\| ^2\bigl\| \nabla F(\bar{\boldsymbol{x}}^k) \bigr\| ^2
.
\end{aligned}
\end{equation}
\begin{equation}\label{eq:y_vgrad_bound}
\begin{aligned}
&\left\|
\boldsymbol{y}^k
- \boldsymbol{v}\nabla F(\bar{\boldsymbol{x}}^k)
\right\|^2\\
\leqslant\;&
2 \|\boldsymbol{e}_{y,k}\|^2 +
\frac{2\|\boldsymbol{v}\|^2}{N}(AL_0+BL_1 b+BL_1\ell G )^2
\|\boldsymbol{e}_{x,k}\|^2.
\end{aligned}
\end{equation}
\end{lemma}

\begin{proof}
We first prove \eqref{eq:alpha_y_bound}. By definition,
\[
\left\| \boldsymbol{\alpha }_k\boldsymbol{y}^k \right\| ^2
= \sum_{i=1}^N \bigl\| \alpha _{i}^{k}\boldsymbol{y}_{i}^{k}\bigr\| ^2 .
\]
Adding and subtracting $\bar{\alpha}_k$ and
$\bar{\alpha}_k\boldsymbol{v}\nabla F(\bar{\boldsymbol{x}}^k)$ to each
term on the right hand side of the above equality gives
\begin{align*}
&\bigl\|\alpha_i^k \boldsymbol{y}_i^k\bigr\|^2\\
=&
\bigl\|
(\alpha_i^k-\bar{\alpha}_k)\boldsymbol{y}_i^k
+ \bar{\alpha}_k\bigl(\boldsymbol{y}_i^k
- \boldsymbol{v}\nabla F(\bar{\boldsymbol{x}}^k)\bigr)
+ \bar{\alpha}_k\boldsymbol{v}\nabla F(\bar{\boldsymbol{x}}^k)
\bigr\|^2
\\
\leqslant & 3\bigl\|
(\alpha_i^k-\bar{\alpha}_k)\boldsymbol{y}_i^k\bigr\|^2
+ 3\bigl\|\bar{\alpha}_k(\boldsymbol{y}_i^k
- \boldsymbol{v}\nabla F(\bar{\boldsymbol{x}}^k))\bigr\|^2
\\
&
+ 3\bigl\|\bar{\alpha}_k\boldsymbol{v}\nabla F(\bar{\boldsymbol{x}}^k)\bigr\|^2,
\end{align*}
where we have used the inequality
$\|\boldsymbol{a}+\boldsymbol{b}+\boldsymbol{c}\|^2
\leqslant
3(\|\boldsymbol{a}\|^2+\|\boldsymbol{b}\|^2+\|\boldsymbol{c}\|^2)$.
Summing over $i$ yields
\begin{align*}
&\left\| \boldsymbol{\alpha }_k\boldsymbol{y}^k \right\| ^2\\
\leqslant &
3\sum_{i=1}^N\bigl\|(\alpha_i^k-\bar{\alpha}_k)\boldsymbol{y}_i^k\bigr\|^2
+3\bar{\alpha}_k^2
\sum_{i=1}^N\bigl\|\boldsymbol{y}_i^k
- \boldsymbol{v}\nabla F(\bar{\boldsymbol{x}}^k)\bigr\|^2
\\
&
+3N\bar{\alpha}_k^{2}\|\boldsymbol{v}\|^2
\bigl\| \nabla F(\bar{\boldsymbol{x}}^k)\bigr\| ^2 .
\end{align*}
By Lemma~\ref{Lemma:pre1}, we have
\[
\sum_{i=1}^N{\bigl\| (\alpha _{i}^{k}-\bar{\alpha}_{i}^{k})\boldsymbol{y}_{i}^{k} \bigr\| ^2}\leqslant \bar{\alpha}_{k}^{2}\sum_{i=1}^N{\frac{\| \boldsymbol{v}\|}{v_i}\| \boldsymbol{y}_{i}^{k}-v_i\nabla F(\bar{\boldsymbol{x}}^k)\| ^2}.
\]
The bound in the previous inequality implies
\begin{equation}
    \begin{aligned}
        \left\| \boldsymbol{\alpha }_k\boldsymbol{y}^k \right\| ^2 \leqslant& \,6\bar{\alpha}_{k}^{2}\sum_{i=1}^N{\frac{\| \boldsymbol{v}\| ^2}{v_{i}^{2}}\bigl\| \boldsymbol{y}_{i}^{k}-v_i\nabla F(\bar{\boldsymbol{x}}^k) \bigr\| ^2}\\
        &+3N\bar{\alpha}_{k}^{2}\| \boldsymbol{v}\| ^2\bigl\| \nabla F(\bar{\boldsymbol{x}}^k) \bigr\| ^2
\\
\leqslant & \, 6\kappa _{v}^{2}\bar{\alpha}_{k}^{2}\bigl\| \boldsymbol{y}^k-\boldsymbol{v}\nabla F(\bar{\boldsymbol{x}}^k) \bigr\| ^2 \\ &+ 3N\bar{\alpha}_{k}^{2}\| \boldsymbol{v}\| ^2\bigl\| \nabla F(\bar{\boldsymbol{x}}^k) \bigr\| ^2,
\end{aligned}
\end{equation}
where $\kappa _v=\max \left\{ \frac{\| \boldsymbol{v}\|}{v_i} \middle| i\in N \right\} 
$.

Noting
$\sum_{i=1}^N\|\boldsymbol{y}_i^k
- v_i \nabla F(\bar{\boldsymbol{x}}^k)\|^2
=\|\boldsymbol{y}^k
- \boldsymbol{v}\nabla F(\bar{\boldsymbol{x}}^k)\|^2$
and $\bar{\alpha}_k\leqslant \alpha$, we obtain \eqref{eq:alpha_y_bound}.

We now prove \eqref{eq:y_vgrad_bound}. We first decompose
$\boldsymbol{y}^k-\boldsymbol{v}\nabla F(\bar{\boldsymbol{x}}^k)$ as
\[
\boldsymbol{y}^k-\boldsymbol{v}\nabla F(\bar{\boldsymbol{x}}^k)
=
\bigl(\boldsymbol{y}^k-\boldsymbol{v}\bar{\boldsymbol{y}}^k\bigr)
+ \boldsymbol{v}\bigl(\bar{\boldsymbol{y}}^k
- \nabla F(\bar{\boldsymbol{x}}^k)\bigr).
\]
Applying the inequality
$\|\boldsymbol{a}+\boldsymbol{b}\|^2
\leqslant 2\|\boldsymbol{a}\|^2+2\|\boldsymbol{b}\|^2$ yields
\begin{equation}\label{Eq:y_minus_v_gradF}
\begin{aligned}
&\left\| \boldsymbol{y}^k-\boldsymbol{v}\nabla F(\bar{\boldsymbol{x}}^k)
\right\| ^2\\
\leqslant&
2\left\| \boldsymbol{y}^k-\boldsymbol{v}\bar{\boldsymbol{y}}^k \right\| ^2
+2\|\boldsymbol{v}\|^2
\left\| \bar{\boldsymbol{y}}^k-\nabla F(\bar{\boldsymbol{x}}^k) \right\| ^2.
\end{aligned}
\end{equation}

By definition, $\boldsymbol{e}_{y,k}
=\boldsymbol{y}^k-\boldsymbol{v}\bar{\boldsymbol{y}}^k$, so the first term
on the right hand side of \eqref{Eq:y_minus_v_gradF} equals $2\|\boldsymbol{e}_{y,k}\|^2$. For the second term on the right hand side of \eqref{Eq:y_minus_v_gradF}, we note
$\bar{\boldsymbol{y}}^k=\frac{1}{N}\sum_{i=1}^{N}\nabla f_i(\boldsymbol{x}_i^k)$ and write
\[
\left\| \bar{\boldsymbol{y}}^k-\nabla F(\bar{\boldsymbol{x}}^k) \right\| ^2
\leqslant \frac{1}{N}
\sum_{i=1}^N
\left\| \nabla f_i(\boldsymbol{x}_i^k)
- \nabla f_i(\bar{\boldsymbol{x}}^k) \right\| ^2.
\]
Using Assumption~\ref{Assum:Lipschitz} and
Assumption~\ref{Assum:gradient_dissimilarity}, together with Lemma~\ref{Lemma:Bound_trajectory}, we can obtain
\begin{equation}\label{Eq:bar_y_minus_gradF}
\left\| \bar{\boldsymbol{y}}^k-\nabla F(\bar{\boldsymbol{x}}^k) \right\| ^2
\leqslant
\frac{\bigl(AL_0+BL_1 b+BL_1\ell G\bigr)^2}{N}
\|\boldsymbol{e}_{x,k}\|^2 .
\end{equation}
Substituting (\ref{Eq:bar_y_minus_gradF}) into (\ref{Eq:y_minus_v_gradF}) gives
\eqref{eq:y_vgrad_bound}, which completes the proof.
\end{proof}

\subsection{Proof of Lemma~\ref{Th:consensus}}

From the inequality in~\eqref{Eq:norm_e_x}, we have
\begin{equation}
\begin{aligned}
\| \boldsymbol{e}_{x,k+1}\|_R^{2}
\leqslant~&(1+\eta)\sigma_R^{2}\|\boldsymbol{e}_{x,k}\|_R^{2} \\
&+ (1+\tfrac{1}{\eta})\|\boldsymbol{I}-\tfrac{\boldsymbol{1}\boldsymbol{u}^{\top}}{N}\|_R^{2}\delta_{R,2}^{2}
\|\boldsymbol{\alpha}_k\boldsymbol{y}^k\|^{2}.
\end{aligned}
\label{eq:expansion-contract}
\end{equation}

Next, by invoking Lemma~\ref{Lemma:pre2}, under
$\alpha \leqslant 
\tfrac{(1-\sigma_R^{2})\sqrt{N}}
{6\sqrt{2} \|\boldsymbol{v}\|\kappa _v\|\boldsymbol{I}-\tfrac{\boldsymbol{1}\boldsymbol{u}^{\top}}{N}\|_R\delta_{R,2}
\left(AL_0+BL_1b+BL_1\ell G\right)}$,
and
$
\eta = \frac{1-\sigma_R^{2}}{3\sigma_R^{2}},
$
we can further bound~\eqref{eq:expansion-contract} as
\begin{equation}
\begin{aligned}
&\|\boldsymbol{e}_{x,k+1}\|_R^{2}
\leqslant \tfrac{1+\sigma_R^{2}}{2}\|\boldsymbol{e}_{x,k}\|_R^{2} \\
&+\frac{12(1+2\sigma_R^{2})
\|\boldsymbol{I}-\tfrac{\boldsymbol{1}\boldsymbol{u}^{\top}}{N}\|_R^{2}\delta_{R,2}^{2}\kappa_v^2}
{1-\sigma_R^{2}}
\,\bar{\alpha}_{k}^{2}\|\boldsymbol{e}_{y,k}\|_C^{2}\\
&+\frac{3N(1+2\sigma_R^{2})
\|\boldsymbol{I}-\tfrac{\boldsymbol{1}\boldsymbol{u}^{\top}}{N}\|_R^{2}
\delta_{R,2}^{2}\|\boldsymbol{v}\|^{2}}
{1-\sigma_R^{2}}
\,\bar{\alpha}_{k}^{2}\|\nabla F(\bar{\boldsymbol{x}}^k)\|^{2}.
\end{aligned}
\label{eq:final-bound}
\end{equation}

For notational convenience, the inequality~\eqref{eq:final-bound} can be written compactly as
\begin{equation}
\begin{aligned}
\|\boldsymbol{e}_{x,k+1}\|_R^{2}
\leqslant &
\tfrac{1+\sigma_R^{2}}{2}\|\boldsymbol{e}_{x,k}\|_R^{2}
+\alpha^{2}\mathcal{C}_{x,1}\|\boldsymbol{e}_{y,k}\|_C^{2}\\
&+\alpha\,\mathcal{C}_{x,2}\bar{\alpha}_{k}\|\nabla F(\bar{\boldsymbol{x}}^k)\|^{2},
\end{aligned}
\end{equation}
where the constants $\mathcal{C}_{x,1}$ and $\mathcal{C}_{x,2}$ are defined in \eqref{Eq:C_x1} and \eqref{Eq:C_x2}.

Next, we analyze $\|\boldsymbol{e}_{y,k}\|_C^{2}$. 

From (\ref{Eq:norm_e_y}), we have
\begin{equation}\label{Eq:norm_e_y_evolution}
\begin{aligned}
& \left\| \boldsymbol{e}_{y,k+1} \right\| _C^{2} \\
\leqslant &(1+\eta )\sigma _{C}^{2}\| \boldsymbol{e}_{y,k}\| _C^{2}\\
& +(1+\tfrac{1}{\eta})\left\| \boldsymbol{I}-\tfrac{\boldsymbol{v}\boldsymbol{1}^{\top}}{N} \right\| _C^{2}\| \nabla f(\boldsymbol{x}^{k+1})-\nabla f(\boldsymbol{x}^k)\| _C^{2}.\\
    \end{aligned}
\end{equation}
For the second term on the right hand side of \eqref{Eq:norm_e_y_evolution}, from (\ref{Eq:grad_diff}), we have
\begin{equation}\label{Eq:grad_diff_repeat}
    \begin{aligned}
        &\| \nabla f(\boldsymbol{x}^{k+1})-\nabla f(\boldsymbol{x}^k)\| ^2 \\
        \leqslant& 2\left( AL_0+BL_1b+BL_1\ell G \right) ^2 \times\\
 &\left( \| \boldsymbol{x}^{k+1}-1\bar{\boldsymbol{x}}^k\| ^2+\| \boldsymbol{x}^k-1\bar{\boldsymbol{x}}^k\| ^2 \right),
    \end{aligned}
\end{equation}
Combining (\ref{Eq:e_y_2.1}) and Lemma~\ref{Lemma:pre2}, we have
\begin{equation}\label{Eq:xy_relation}
    \begin{aligned}
&\| \boldsymbol{x}^{k+1}-1\bar{\boldsymbol{x}}^k\| ^2\\
\leqslant& 2\left\| \boldsymbol{R}-\frac{1u^{\top}}{N} \right\| _{R}^{2}\| \boldsymbol{x}^k-1\bar{\boldsymbol{x}}^k\| _{R}^{2}+2\| \boldsymbol{\alpha }_k\boldsymbol{y}^k\| ^2
\\
\leqslant& \left( 2\sigma _{R}^{2}+\frac{24\bar{\alpha}_{k}^{2}\| \boldsymbol{v}\| ^2 \kappa_v^2 \left( AL_0+BL_1b+BL_1\ell G \right) ^2}{N} \right) \| \boldsymbol{e}_{x,k}\| _R^{2}
\\
&+24\kappa_v^2\bar{\alpha}_{k}^{2}\| \boldsymbol{e}_{y,k}\| _C^{2}+6N\| \boldsymbol{v}\| ^2\bar{\alpha}_{k}^{2}\| \nabla F(\bar{\boldsymbol{x}}^k)\| ^2.
\end{aligned}
\end{equation}

Substituting~\eqref{Eq:grad_diff_repeat} and~\eqref{Eq:xy_relation} into~\eqref{Eq:norm_e_y_evolution}, and using the condition $\alpha \leqslant \tfrac{1-\sigma _{C}^{2}}{12\sqrt{2} \kappa_v \delta_{C,2}\| \boldsymbol{I}-\frac{v\mathbf{1}^{\top}}{N} \| _{C}(AL_0+BL_1 b+BL_1\ell G )}
$, we obtain
\begin{equation}
\begin{aligned}
\|\boldsymbol{e}_{y,k+1}\|_C^{2}
\leqslant~&
\tfrac{1+\sigma_{C}^{2}}{2}\|\boldsymbol{e}_{y,k}\|_C^{2}\\
&+\mathcal{C}_{y,1}\|\boldsymbol{e}_{x,k}\|_R^{2}
+\alpha\,\mathcal{C}_{y,2}\bar{\alpha}_{k}
\|\nabla F(\bar{\boldsymbol{x}}^k)\|^{2},
\end{aligned}
\label{Eq:final_e_y}
\end{equation}
where the constants $\mathcal{C}_{y,1}$ and $\mathcal{C}_{y,2}$ are defined in \eqref{Eq:C_y1} and \eqref{Eq:C_y2}.

\section*{Appendix G: Proof of Lemma~\ref{Th:Descent}}\label{Appendix:proof_Th:Descent}

We now derive a recursive descent relation for the global objective function \( F(\bar{\boldsymbol{x}}^k) \).  
From~\eqref{Eq:temp15}, when the stepsize satisfies 
$\alpha \leqslant \frac{\boldsymbol{u}^{\top}\boldsymbol{v}}{9LN\|\boldsymbol{v}\|^2}$, we have
\begin{equation}\label{Equ:descent1}
\begin{aligned}
&\frac{\boldsymbol{u}^{\top}\boldsymbol{v}}{3N}\bar{\alpha}_k\| \nabla F(\bar{\boldsymbol{x}}^k)\| ^2\leqslant F(\bar{\boldsymbol{x}}^k)-F(\bar{\boldsymbol{x}}^{k+1})
\\
& \quad +\left( \frac{6L\kappa _{uv}\| \boldsymbol{v} \|}{2N}+\frac{2\kappa _{uv}^{2}\| \boldsymbol{v}\| ^2}{N\boldsymbol{u}^{\top}\boldsymbol{v}} \right) \bar{\alpha}_k\| \boldsymbol{y}^k-\boldsymbol{v}\nabla F(\bar{\boldsymbol{x}}^k)\| ^2.
\end{aligned}
\end{equation} 

By Lemma~\ref{Lemma:pre2}, we have
\begin{equation}\label{Eq:y^k-vnablaF}
\begin{aligned}
     &\| \boldsymbol{y}^k - \boldsymbol{v}\nabla F(\bar{\boldsymbol{x}}^k) \|^2 \\
\leqslant& 2 \|\boldsymbol{e}_{y,k}\|_C^2 + \frac{2\|\boldsymbol{v}\|^2}{N}(AL_0 + BL_1(b + \ell G))^2 \|\boldsymbol{e}_{x,k}\|_R^2.  
\end{aligned}
\end{equation}
Substituting \eqref{Eq:y^k-vnablaF}  into \eqref{Equ:descent1} yields the desired result.

\bibliographystyle{ieeetr}
\bibliography{reference}

\end{document}